\documentclass[reqno]{amsart}

\usepackage{amsmath}
\usepackage{amsfonts}
\usepackage{amssymb,enumerate}
\usepackage{amsthm}
\usepackage[all]{xy}
\usepackage{rotating}
\usepackage{hyperref}

\theoremstyle{plain}
\newtheorem{lem}{Lemma}[section]
\newtheorem{cor}[lem]{Corollary}
\newtheorem{prop}[lem]{Proposition}
\newtheorem{thm}[lem]{Theorem}

\theoremstyle{definition}
\newtheorem{defn}[lem]{Definition}
\newtheorem{ex}[lem]{Example}
\newtheorem{question}[lem]{Question}

\newtheorem{disc}[lem]{Remark}

\newtheorem{notn}[lem]{Notation}
\newtheorem{fact}[lem]{Fact}

\newcommand{\cat}[1]{\mathcal{#1}}

\newcommand{\catd}{\cat{D}}

\newcommand{\pd}{\operatorname{pd}}

\newcommand{\rank}{\operatorname{rank}}	
\newcommand{\amp}{\operatorname{amp}}

\newcommand{\ann}{\operatorname{Ann}}
\newcommand{\mspec}{\operatorname{m-Spec}}

\newcommand{\lotimes}{\otimes^{\mathbf{L}}}
\newcommand{\HH}{\operatorname{H}}
\newcommand{\Hom}{\operatorname{Hom}}	

\newcommand{\spec}{\operatorname{Spec}}

\newcommand{\tor}{\operatorname{Tor}}

\newcommand{\shift}{\mathsf{\Sigma}}

\newcommand{\ideal}[1]{\mathfrak{#1}}
\newcommand{\m}{\ideal{m}}

\newcommand{\p}{\ideal{p}}
\newcommand{\q}{\ideal{q}}

\newcommand{\fa}{\ideal{a}}
\newcommand{\fb}{\ideal{b}}
\newcommand{\fc}{\ideal{c}}

\newcommand{\comp}[1]{\widehat{#1}}
\newcommand{\ol}{\overline}

\newcommand{\supp}{\operatorname{supp}}
\newcommand{\Supp}{\operatorname{Supp}}

\newcommand{\VE}{\operatorname{V}}
\newcommand{\cosupp}{\operatorname{co-supp}}

\newcommand{\bbz}{\mathbb{Z}}

\newcommand{\xra}{\xrightarrow}

\newcommand{\res}{\xra{\simeq}}

\newcommand{\x}{\underline{x}}

\renewcommand{\geq}{\geqslant}
\renewcommand{\leq}{\leqslant}

\newcommand{\Ext}[4][R]{\operatorname{Ext}_{#1}^{#2}(#3,#4)}	
\newcommand{\Rhom}[3][R]{\mathbf{R}\!\operatorname{Hom}_{#1}(#2,#3)}	
\newcommand{\Lotimes}[3][R]{#2\otimes^{\mathbf{L}}_{#1}#3}
\newcommand{\Otimes}[3][R]{#2\otimes_{#1}#3}
\renewcommand{\Hom}[3][R]{\operatorname{Hom}_{#1}(#2,#3)}

\newcommand{\Jac}{\operatorname{J}}

\newcommand{\LL}[2]{\mathbf{L}\Lambda^{\ideal{#1}}(#2)}
\newcommand{\LLL}[2]{\mathbf{L}\widehat\Lambda^{\ideal{#1}}(#2)}
\newcommand{\RG}[2]{\mathbf{R}\Gamma_{\ideal{#1}}(#2)}

\newcommand{\Comp}[2]{\widehat{#1}^{\ideal{#2}}}
\newcommand{\rad}[1]{\operatorname{rad}(#1)}

\newcommand{\ssm}{\smallsetminus}

\newcommand{\Cosupp}{\operatorname{Co-supp}}
\newcommand{\catdfb}{\catd_{\text{b}}^{\text{f}}}
\newcommand{\catdb}{\catd_{\text{b}}}
\newcommand{\catdf}{\catd^{\text{f}}}
\numberwithin{equation}{lem}

\begin{document}

\bibliographystyle{amsplain}

\author{Sean Sather-Wagstaff}

\address{Sean Sather-Wagstaff, Department of Mathematics,
NDSU Dept \# 2750,
PO Box 6050,
Fargo, ND 58108-6050
USA}

\email{sean.sather-wagstaff@ndsu.edu}

\urladdr{http://www.ndsu.edu/pubweb/\~{}ssatherw/}

\thanks{
Sean Sather-Wagstaff was supported in part by a grant from the NSA}

\author{Richard Wicklein}

\address{Richard Wicklein, Mathematics and Physics Department, MacMurray College, 447 East College Ave., Jacksonville, IL 62650, USA}

\email{richard.wicklein@mac.edu}

\title{Support and adic finiteness for complexes}

\date{\today}



\begin{abstract}
Let $X$ be a chain complex over a commutative noetherian ring $R$, that is, an object in the derived category $\catd(R)$.
We investigate the small support and co-support of $X$, introduced by Foxby and Benson, Iyengar, and Krause.
We show that the derived functors $\Lotimes M-$ and $\Rhom M-$ can detect isomorphisms in $\catd(R)$
between complexes with restrictions on their supports or co-supports.
In particular, 
the derived local (co)homology functors $\RG a-$ and $\LL a-$ with respect to an ideal $\fa\subsetneq R$ have the same ability.
Furthermore, we give reprove some results of Benson, Iyengar, and Krause in our setting, with more direct proofs.
Also, we 
include some computations of co-supports, since this construction is still quite mysterious.
Lastly, we investigate ``$\fa$-adically finite'' $R$-complexes, that is, the $X\in\catdb(R)$ that are $\fa$-cofinite \textit{\`a la} Hartshorne.
For instance, we characterize these complexes in terms of a finiteness condition on $\LL aX$.
\end{abstract}

\maketitle


\section{Introduction} \label{sec130805a}
Throughout this paper let $R$ be a commutative noetherian ring, let $\fa \subsetneq R$ be a proper ideal of $R$, and let $\Comp{R}{a}$ be the $\fa$-adic completion of $R$.
We work in the derived category $\catd(R)$ the objects of which are the $R$-complexes, indexed homologically
$$X=\cdots\to X_{i+1}\xra{\partial^X_{i+1}}X_i\to\cdots.$$
We let $\Lambda^{\fa}(-)$ denote the $\fa$-adic completion functor, and
$\Gamma_{\fa}(-)$ is the $\fa$-torsion functor.
The associated left- and right-derived functors, respectively, are $\LL a-$ and $\RG a-$.
(See Section~\ref{sec140109b} for some background information on these topics.)
The left- and right-derived functors of $\Otimes --$ and $\Hom --$ are denoted
$\Lotimes --$ and $\Rhom --$.

\

We investigate the small support $\supp_R(X)$ of Foxby~\cite{foxby:bcfm}
and the co-support $\cosupp_R(X)$ of Benson, Iyengar, and Krause~\cite{benson:csc}; see Definitions~\ref{def120925c} and~\ref{defn130529a}.
Sections~\ref{sec130805b} and~\ref{sec140110a} contain alternate characterizations of these subsets 
(see Propositions~\ref{prop140111d} and~\ref{prop140111c})
and some of their basic properties.

Section~\ref{sec140710a} is devoted to some consequences of support conditions for morphisms.
For instance, we show in the next result  that a restriction of the small support or co-support  is strong enough
to guarantee that the derived functors $\Lotimes M-$ and $\Rhom M-$ can detect isomorphisms in $\catd(R)$.
It is contained in Theorems~\ref{thm130318aqq} and~\ref{thm130318aq} below.
(See also Corollaries~\ref{thm130318aqqc} and~\ref{thm130318aqc} for some  consequences
for Koszul homology and local (co)homology.)

\begin{thm}\label{thm130318ax}
Let $M\in\catd(R)$, and let
$f\colon Y \to Z$ be morphism in $\catd(R)$ such that
$\supp_R(Y), \supp_R(Z) \subseteq \supp_R(M)$
or
$\cosupp_R(Y), \cosupp_R(Z) \subseteq \supp_R(M)$.
Then the following conditions are equivalent:
\begin{enumerate}[\rm(i)]
\item 
$f$ is an isomorphism in $\catd(R)$;
\item 
$\Rhom{M}{f}$ is an isomorphism in $\catd(R)$;
\item 
$\Lotimes{M}{f}$ is an isomorphism in $\catd(R)$.
\end{enumerate}
\end{thm}

Also in this section, we recover (with more direct proofs)
results of Benson, Iyengar, and Krause~\cite{benson:lcstc,benson:csc}
that use (co-)support to characterize the $R$-complexes $X$ such that one of the natural morphisms
$\RG aX\to X\to\LL aX$ is an isomorphism. See Propositions~\ref{prop130619ad} and~\ref{prop130619ae} for our results
and, e.g., \cite{dwyer:cmtm, yekutieli:ccc, yekutieli:hct, yekutieli:sccmc} for more on these complexes.

In Section~\ref{sec140710b}, we give some computations of $\cosupp_R(X)$, in part, because this construction is 
not as well understood as $\supp_R(X)$.
For instance, the next result is contained in
Theorems~\ref{prop140118d} and~\ref{prop140709a}.

\begin{thm}\label{prop140118dx}
Assume that $R$ has a dualizing complex.
\begin{enumerate}[\rm(a)]
\item\label{prop140118dxa}
For  each $X\in\catdfb(R)$, one has
$$\cosupp_R(X)=\supp_R(X)\bigcap\cosupp_R(R)\subseteq\supp_R(X).$$
\item\label{prop140118dxb}
If, moreover, $R$ is a 1-dimensional  integral domain, then 
$$\cosupp_R(R)=\begin{cases}
\mspec(R)&\text{if $R$ is local and  complete, and}\\
\spec(R)&\text{otherwise.}\end{cases}$$
\end{enumerate}
\end{thm}

The paper concludes with Section~\ref{sec140109c}, wherein we investigate ``$\fa$-adically finite'' complexes;
see Definition~\ref{def120925d}.
This notion originates with work of Hartshorne~\cite{hartshorne:adc} and continues, e.g.,
in~\cite{delfino:cmlc,hartshorne:adc,kawasaki:ccma,kawasaki:ccc,melkersson:mci}.
To allow for some flexibility in the study of such complexes, we prove the following result in Theorem~\ref{thm130612a}. 
Here $\LLL aX$ is constructed like $\LL aX$, but considered as a functor from $\catd(R)\to\catd(\Comp Ra)$; see
the beginning of Section~\ref{sec140109b}.

\begin{thm}\label{cor130612a}
Let $X\in\catd_{\text b}(R)$. Then the following conditions are equivalent:
\begin{enumerate}[\rm(i)]
\item\label{cor130612a1}
One has $\Lotimes{K(\underline{x})}{X}\in\catdfb(R)$  for some (equivalently for every) generating sequence $\underline{x}$ of $\fa$;
\item\label{cor130612a2}
One has  $\Lotimes{R/\mathfrak{a}}{X}\in\catd^{\text{f}}(R)$;
\item\label{cor130612a3}
One has  $\Rhom{R/\mathfrak{a}}{X}\in\catd^{\text{f}}(R)$;
\item\label{cor130612a4}
One has $\LLL aX\in\catdfb(\Comp Ra)$.
\end{enumerate}
\end{thm}

While one may not be surprised by the equivalence of conditions~\eqref{cor130612a1}--\eqref{cor130612a3} in this result,
we did not expect them to be equivalent to condition~\eqref{cor130612a4}.
Another interesting feature  of this result is the use of techniques from differential graded algebra in the proofs of
the implications $\eqref{cor130612a2}\implies\eqref{cor130612a1}$ and $\eqref{cor130612a3}\implies\eqref{cor130612a1}$;
see Propositions~\ref{prop130610a} and~\ref{prop130613b}.

Lastly, it is worth noting that many of the results in this paper are tools for use in the sequel~\cite{sather:ascfc}.

\section{Background}\label{sec140109b} 

\subsection*{Derived Categories}
Standard references on this subject include~\cite{gelfand:moha,hartshorne:rad, verdier:cd, verdier:1}.

The quantities $\inf(X)$ and $\sup (X)$ are the infimum and supremum, respectively,
of the set $\{i\in\bbz\mid\HH_i(X)\neq 0\}$,
and $\amp(X):=\sup(X)-\inf(X)$.
Given an integer $i$, we let $\shift^iX$ denote the $i$th \emph{shift}
(or \emph{suspension}) of $X$.
Isomorphisms in $\catd(R)$ are identified by the symbol $\simeq$.

We let $\catd_+(R)$ denote the full subcategory of $\catd(R)$ consisting of complexes $X$ such that $\inf(X)>-\infty$, that is such that
$\HH_i(X)=0$ for $i\ll 0$. 
We let $\catd_-(R)$ denote the full subcategory of $\catd(R)$ consisting of complexes $X$ such that $\sup(X)<\infty$, that is such that
$\HH_i(X)=0$ for $i\gg 0$, and set $\catdb(R):=\catd_+(R)\bigcap\catd_-(R)$. 
We let $\catdf(R)$ denote the full subcategory of $\catd(R)$ consisting of complexes $X$ such that each homology module
$\HH_i(X)$ is finitely generated. 
For each $\star\in\{+,-,b\}$ we set $\catdf_\star(R):=\catdf(R)\bigcap\catd_{\star}(R)$.

An $R$-complex $F$ is \emph{semiflat} if the functor $\Otimes F-$, defined on the category of $R$-complexes, respects
injective quasiisomorphisms. 
(These are the ``DG-flat'' complexes of~\cite{avramov:hdouc}.)
A \emph{semiflat resolution} of an $R$-complex $X$ is a quasiisomorphism $F\xra\simeq X$ 
with $F$ semiflat, and one defines $\LL aX:=\Lambda^{\fa}(F)$ and
$\Lotimes XY:=\Otimes FY$ for each $R$-complex $Y$.
Every $R$-complex admits a semiflat resolution by~\cite[1.5 and 1.6]{avramov:hdouc},
and $\LL a-$ and $\Lotimes --$ define  (bi)functors on $\catd(R)$.
Since the complex $\Lambda^{\ideal{a}}(F)$ consists of $\Comp Ra$-modules and $\comp Ra$-module homomorphisms,
this also defines a functor $\catd(R)\to\catd(\Comp Ra)$ which we denote $\LLL a-$, following~\cite{shaul:tpdcf}.
This is well-defined by~\cite[Section~1]{lipman:lhcs}.
Moreover, if $F\colon \catd(\Comp Ra)\to\catd(R)$ is the forgetful functor, then we have
a natural isomorphism
$F\circ\mathbf{L}\comp\Lambda^{\fa}\simeq\mathbf{L}\Lambda^{\fa}$.

An $R$-complex $I$ is \emph{semiinjective} if the functor $\Hom -I$ converts injective quasiisomorphisms
into surjective quasiisomorphisms.
(These are the ``DG-injec-tive'' complexes of~\cite{avramov:hdouc}.)
A \emph{semiinjective resolution} of an $R$-complex $X$ is a quasiisomorphism $X\xra\simeq I$ 
with $I$ semiinjective, and one defines 
$\RG aX:=\Gamma_{\fa}(I)$
and $\Rhom YX:=\Hom YI$ for each $R$-complex $Y$.
Every $R$-complex admits a semiinjective resolution by~\cite[1.6]{avramov:hdouc}
and $\RG a-$ and $\Rhom --$ define  (bi)functors on $\catd(R)$.

\subsection*{Derived Local (Co)homology}
These notions originate in~\cite{hartshorne:rad,hartshorne:lc}, and are developed extensively, e.g., 
in~\cite{lipman:lhcs,benson:lcstc, foxby:daafuc, greenlees:dfclh, lipman:llcd}.

\begin{fact}\label{fact130619b}
Let $X\in\catd_+(R)$.
If $X\in\catdf_+(R)$, then there is a natural isomorphism
$\LL aX\simeq \Lotimes{\Comp Ra}{X}$ by~\cite[Proposition 2.7]{frankild:volh}.

Let $\x=x_1,\ldots,x_n$ be a generating sequence for $\fa$. Then 
$\RG aR$ is isomorphic in $\catd(R)$ to the \v{C}ech complex  $\check{C}(\x)$.
It follows that $\pd_R(\RG aR)<\infty$.
Indeed, the \v Cech complex $\check{C}(\x)$ is a bounded complex of direct sums of modules
of the form $R_s\cong R[T]/(sT-1)$ with $s\in R$. 
Since $R_s$ has a projective resolution
$$0\to R[T]\xra{sT-1}R[T]\to R_s\to 0$$
we conclude that $\pd_R(R_s)\leq 1$, hence $\pd_R(\RG aR)<\infty$.
\end{fact}

\begin{fact}\label{fact130619b'}
By~\cite[(0.3)$_{\text{aff}}$]{lipman:lhcs} and~\cite[Proposition 3.1.2]{lipman:llcd}, there are  isomorphisms
\begin{gather*}
\RG a-\simeq\Lotimes{\RG aR}{-}
\qquad\qquad\qquad
\LL a-\simeq\Rhom{\RG aR}{-}.
\end{gather*}
of functors on $\catd(R)$.
\end{fact}

\subsection*{Minimal Injective Resolutions}
For a module, the notion of a minimal injective resolution is standard. 
For complexes, one may consult~\cite{avramov:dgha,chen:sirccr}. 

\begin{defn}\label{defn130618c}
Let $X\in\catd(R)$. An injective resolution $X\res J$ is \textit{minimal} if for all $i$ the kernel of the differential $\partial_{i}^{J}: J_{i} \to J_{i+1}$ is an essential submodule of $J_i$.
\end{defn}

\begin{fact}\label{fact130622b}
Every $X\in\catd_-(R)$ has a minimal injective resolution $X\xra\simeq J$, and every such resolution satisfies
$J_i=0$ for all $i>\sup(X)$. 
If $S$ is a multiplicatively closed subset of $R$, then
the localization $S^{-1}X\xra\simeq S^{-1}J$ is a
minimal injective resolution over $S^{-1}R$.
Also, the induced morphism $\RG aX\res\Gamma_\fa(J)$ is a minimal injective resolution over $R$.
\end{fact}

\section{Support}\label{sec130805b} 

The point of this section is to investigate some useful aspects of 
Foxby's notion of support for complexes from~\cite{foxby:bcfm}. 
One main result is Proposition~\ref{prop140111d}.

\begin{notn}
Set $\VE(\fa):=\{\p\in\spec(R)\mid\fa\subseteq\p\}$.
For each $\p\in\spec(R)$, set $\kappa(\p):=R_{\p}/\p R_{\p}$.
Given a sequence $\x=x_1,\ldots,x_n\in R$ the Koszul complex on $\x$ is denoted $K^R(\x)$.
\end{notn}

\begin{defn}
\label{def120925c}
Let $X\in\catd(R)$.
\begin{enumerate}[(a)]
\item
The ``small,'' or ``homological,'' support of $X$ is 
$$\operatorname{supp}_R(X):=\{ \mathfrak{p} \in \spec(R)\mid \Lotimes{\kappa(\mathfrak{p})}{X} \not\simeq 0\}.$$  
\item
The ``large'' support of $X$ is 
$$\operatorname{Supp}_R(X):=\{\mathfrak{p} \in \spec(R)\mid  X_{\mathfrak{p}} \not\simeq 0\}.$$
\end{enumerate}
\end{defn}

\begin{fact}\label{defn130503a}
If $M$ is an $R$-module, then
\begin{align*}
\operatorname{supp}_R(M)&=\{ \mathfrak{p} \in \operatorname{Spec}(R)\mid  \tor^R_i (R/ \mathfrak{p}, M)_{\mathfrak{p}}\neq 0 \hspace{.05in}  \text{for some} \hspace{.05in} i \}\\
\operatorname{Supp}_R(M) &= \{ \mathfrak{p} \in \operatorname{Spec}(R)\mid  M_{\mathfrak{p}} \neq 0\}.
\end{align*}
\end{fact}

\begin{fact}\label{fact140109a}
Let $X\in\catd(R)$.
It is straightforward to show that
$$\supp_R(X)\subseteq\Supp_R(X)=\bigcup_{i\in\bbz}\Supp_R(\HH_i(X)).$$
It takes significantly more work to show that
$X \simeq 0$ if and only if $\supp_R(X) = \emptyset$;
see~\cite[5.2, 9.2]{benson:lcstc}.

If $X\in\catd_+^{\text{f}}(R)$, then  Nakayama's Lemma implies that $\supp_R(X)=\Supp_R(X)$.
In particular, if $\x=x_1,\ldots,x_n$ is a generating sequence for $\fa$, then we have
$\supp_R(K^R(\x))=\supp_R(R/\fa)=\VE(\fa)$.
\end{fact}

The next fact is the key to our alternate characterization of small support.

\begin{fact}[\protect{\cite[2.1, 4.1]{foxby:daafuc}}]\label{prop140111a}
Assume that $(R,\m,k)$ is local. Let $\x=x_1,\ldots,x_n$ be a generating sequence for $\m$, and let $X\in\catd(R)$.
Then the following conditions  are equivalent:
\begin{enumerate}[\rm(i)]
\item \label{prop140111a1}
$\m\in\supp_R(X)$, i.e., $\Lotimes kX\not\simeq 0$;
\item \label{prop140111a2}
$\Lotimes{K^{R}(\x)}{X}\not\simeq 0$;
\item \label{prop140111a3}
$\mathbf{L}\Lambda^{\m}(X)\not\simeq 0$;
\item \label{prop140111a5}
$\Rhom{k}{X} \not\simeq 0$;
\item \label{prop140111a4}
$\Rhom{K^{R}(\x)}{X}\not\simeq 0$;
\item \label{prop140111a6}
$\mathbf{R}\Gamma_{\m}(X)\not\simeq 0$.
\end{enumerate}
\end{fact}

Part of the following result is in~\cite[(9.2)]{benson:lcstc}; see, however, \cite[Remark 2.3]{chen:sirccr} for some words of caution.

\begin{prop}\label{prop140111d}
Let $X\in\catd(R)$ and  $\p\in\spec(R)$, and let  $\x=x_1,\ldots,x_n$ be a generating sequence for $\p$.
Then the following conditions are equivalent:
\begin{enumerate}[\rm(i)]
\item \label{prop140111d1}
$\p\in\supp_R(X)$, i.e., $\Lotimes {\kappa(\p)}X\not\simeq 0$;
\item \label{prop140111d2}
$\Lotimes{K^{R_\p}(\x)}{X}\not\simeq 0$;
\item \label{prop140111d3}
$\mathbf{L}\Lambda^{\p R_{\p}}(X_{\p})\not\simeq 0$, that is, $\mathbf{L}\Lambda^{\p}(X_{\p})\not\simeq 0$;
\item \label{prop140111d5}
$\Rhom[R_{\p}]{\kappa(\p)}{X_{\p}} \not\simeq 0$;
\item \label{prop140111d4}
$\Rhom[R_{\p}]{K^{R_\p}(\x)}{X_{\p}} \not\simeq 0$;
\item \label{prop140111d6}
$\mathbf{R}\Gamma_{\p R_{\p}}(X_{\p})\not\simeq 0$, that is, $\mathbf{R}\Gamma_{\p}(X_{\p})\not\simeq 0$;
\item \label{prop140111d7}
$\p R_\p\in\supp_{R_\p}(X_\p)$.
\end{enumerate}
\end{prop}

\begin{proof}
By applying Fact~\ref{prop140111a} to the $R_{\p}$-complex $X_{\p}$, we see that the following complexes are simultaneously isomorphic to $0$ in $\catd(R)$.
\begin{gather*}
\Lotimes{\kappa(\mathfrak{p})}{X}
\simeq
\Lotimes{(\Lotimes[R_{\p}]{\kappa(\mathfrak{p})}{R_{\p}})}{X}
\simeq
\Lotimes[R_{\p}]{\kappa(\mathfrak{p})}{X_\p}
\\
\Lotimes{K^{R_\p}(\x)}{X}
\simeq
\Lotimes{(\Lotimes[R_{\p}]{K^{R_\p}(\x)}{R_{\p}})}{X}
\simeq
\Lotimes[R_{\p}]{K^{R_\p}(\x)}{X_\p}
\\
\ \ \
\mathbf{L}\Lambda^{\p R_{\p}}(X_{\p})\simeq\mathbf{L}\Lambda^{\p}(X_{\p})
\qquad\qquad\quad
\mathbf{R}\Gamma_{\p R_{\p}}(X_{\p})\simeq \mathbf{R}\Gamma_{\p}(X_{\p})
\\
\Rhom[R_{\p}]{K^{R_\p}(\x)}{X_{\p}}
\qquad\qquad\qquad
\Rhom[R_{\p}]{\kappa(\p)}{X_{\p}} 
\end{gather*}
(The isomorphisms in the third line are from~\cite[Theorem 6.5]{yekutieli:hct}.)
This gives the equivalence of conditions~\eqref{prop140111d1}--\eqref{prop140111d6}.
The equivalence of conditions~\eqref{prop140111d6} and~\eqref{prop140111d7} follows from a comparison
with conditions~\eqref{prop140111a6} and~\eqref{prop140111a1} of Fact~\ref{prop140111a}.
\end{proof}

\begin{disc}\label{rmk140708a}
The equivalence of conditions~\eqref{prop140111d1} and~\eqref{prop140111d6} 
in Proposition~\ref{prop140111d} show that our definition
of $\supp_R(X)$ is equivalent to that from~\cite{benson:lcstc}; see~\cite[Theorem 9.1]{benson:lcstc}.
\end{disc}

See~\cite[Remark 2.3]{chen:sirccr} for a discussion of what goes wrong in the next result when $X\notin\catd_-(R)$.

\begin{prop}\label{lem130622a}
If $X\in\catd_-(R)$ with minimal injective resolution $X\xra\simeq J$,
then
$$\supp_R(X)=\bigcup_{i\in\bbz}\{\p\in\spec(R)\mid\text{$E_R(R/\p)$ is a summand of $J_i$}\}.$$
\end{prop}

\begin{proof}
By~\cite[2.1]{chen:sirccr}, it suffices to observe that, for each $\p\in\spec(R)$, the complex $J_\p$ is minimal
by Fact~\ref{fact130622b}, and  $J_\p$ is ``homotopically injective'' since it is a bounded above complex of injectives.
\end{proof}

\begin{cor}\label{prop130619a}
Let $X\in\catd_-(R)$ such that $\supp_R(X)\subseteq V(\fa)$. 
\begin{enumerate}[\rm(a)]
\item \label{prop130619a1}
The minimal injective resolution of $X$
consists of $\fa$-torsion modules. 
\item \label{prop130619a4}
Each injective resolution of $X$
consisting of $\fa$-torsion modules is an $\Comp Ra$-complex.
\end{enumerate}
\end{cor}

\begin{proof}
Let $X\xra\simeq J$ be a minimal injective resolution. 

\eqref{prop130619a1} 
By Proposition~\ref{lem130622a}, for each $i$ we have
$J_i\cong\oplus_{\p\in\supp_{R}(X)}E_R(R/\p)^{(\mu^i_\p)}$ for some sets $\mu^i_\p$.
Since each $\p\in\supp_R(X)$ is in $V(\fa)$, it follows that each summand $E_R(R/\p)^{(\mu^i_\p)}$
is $\fa$-torsion, so each $J_i$ is $\fa$-torsion as well.

\eqref{prop130619a4} 
Each module $J_i$ is $\fa$-torsion, so it is an $\Comp Ra$-module by~\cite[Fact 2.1(a)]{kubik:hamm2},
and each differential $\partial^J_i$ is $\Comp Ra$-linear by~\cite[Lemma 2.2(a)]{kubik:hamm2}.
\end{proof}

The next few results document some basic properties of small support. 
Several of these are known, see~\cite[Proposition 5.1 and Theorem 5.6]{benson:lcstc}, however our proofs are more direct. 

\begin{prop}\label{prop140710g}
Given a distinguished triangle $X\to Y\to Z\to$ in $\catd(R)$
one~has%
$$\supp_R(Y)\subseteq\supp_R(X)\bigcup\supp_R(Z).$$
\end{prop}

\begin{proof}
Let $\p\in\spec(R)$.
From the induced triangle
$$\Lotimes{\kappa(\p)}{X}\to\Lotimes{\kappa(\p)}{Y}\to\Lotimes{\kappa(\p)}{Z}\to$$
we conclude that, if $\p\notin\supp_R(X)\bigcup\supp_R(Z)$, then
$\p\notin\supp_R(Y)$, as desired.
\end{proof}

\begin{prop}\label{prop140710a}
Given a set $\{X(i)\}_{i\in\Lambda}\subseteq\catd(R)$
one has 
$$\supp_R\left(\coprod_iX(i)\right)=\bigcup_i\supp_R(X(i))\subseteq\supp_R\left(\prod_iX(i)\right).$$
\end{prop}

\begin{proof}
Given a prime $\p\in\spec(R)$ and en element $j\in\Lambda$, we have isomorphisms
\begin{align*}
\kappa(\p)\lotimes_R\left(\coprod_iX(i)\right)
&\simeq\coprod_i\left(\kappa(\p)\lotimes_RX(i)\right) \\
\kappa(\p)\lotimes_R\left(\prod_iX(i)\right)
&\simeq\left(\kappa(\p)\lotimes_RX(j)\right) \coprod\left(\kappa(\p)\lotimes_R\left(\prod_{i\neq j}X(i)\right)\right).
\end{align*}
The conclusion
$\supp_R\left(\coprod_iX(i)\right)=\bigcup_i\supp_R(X(i))\subseteq\supp_R\left(\prod_iX(i)\right)$
follows readily by definition.
\end{proof}

When $X,Y\in\catd_+(R)$, the next result is from~\cite[Theorem 7.1(c)]{foxby:cim}.

\begin{prop}\label{fact130611a}
If $X,Y\in\catd(R)$, then 
$$\supp_{R}(\Lotimes{X}{Y}) = \supp_R(X) \bigcap \supp_{R}(Y).$$
\end{prop}

\begin{proof}
The isomorphism
\begin{align*}
\Lotimes{\kappa(\p)}{(\Lotimes{X}{Y})}
&\simeq\Lotimes[\kappa(\p)]{(\Lotimes{\kappa(\p)}X)}{(\Lotimes{\kappa(\p)}Y)}
\end{align*}
conspires with the K\"unneth formula to imply that
\begin{align*}
\HH_i(\Lotimes{\kappa(\p)}{(\Lotimes{X}{Y})})
&\cong\bigoplus_{p+q=i}\Lotimes[\kappa(\p)]{\HH_p(\Lotimes{\kappa(\p)}X)}{\HH_q(\Lotimes{\kappa(\p)}Y)}.
\end{align*}
From this, it follows that $\Lotimes{\kappa(\p)}{(\Lotimes{X}{Y})}\not\simeq 0$
if and only if $\Lotimes{\kappa(\p)}X\not\simeq 0$ and $\Lotimes{\kappa(\p)}Y\not\simeq 0$,
as desired.
\end{proof}

\begin{prop}\label{lem130611a}
If $X\in\catd(R)$, then
$\supp_{R}(\RG{a}{X}) = \supp_{R}(X) \bigcap \VE(\fa)$.
In particular, we have $\supp_{R}(\RG{a}{R}) = \VE(\fa)$.
\end{prop}

\begin{proof}
Let $R\xra\simeq I$ be a minimal injective resolution.
It is well known that each injective hull $E_R(R/\p)$ occurs in a summand of some $J_i$;
see also Fact~\ref{fact140109a} and Proposition~\ref{lem130622a}.
From the fact
$$\Gamma_\fa(E_R(R/\p))\cong
\begin{cases}
E_R(R/\p)&\text{if $\p\in\VE(\fa)$} \\
0&\text{if $\p\notin\VE(\fa)$} \end{cases}$$
we conclude that $E_R(R/\p)$ occurs in a summand of some $\Gamma_\fa(J)_i$ if and only if $\p\in\VE(\fa)$.
Since $\RG aR\xra\simeq\Gamma_\fa(J)$ is minimal injective resolution by Fact~\ref{fact130622b}, we have
$\supp_{R}(\RG{a}{R})=\VE(\fa)$
by Proposition~\ref{lem130622a}.
From this, we have
\begin{align*}
\supp_{R}(\RG{a}{X}) 
&= \supp_{R}(\Lotimes{\RG{a}{R}}{X}) \\
&= \supp_{R}(X) \bigcap \supp_{R}(\RG{a}{R})\\
&=\supp_{R}(X) \bigcap \VE(\fa)
\end{align*}
by Fact~\ref{fact130619b'} and Proposition~\ref{fact130611a}.
\end{proof}

\begin{prop}\label{prop140709b}
Let $X\in\catd(R)$.
Then the sets $\supp_R(X)$ and $\Supp_R(X)$ have the same minimal elements with respect to containment, that is,
we have $\min(\supp_R(X))=\min(\Supp_R(X))$. 
\end{prop}

\begin{proof}
For the containment $\min(\supp_R(X))\supseteq\min(\Supp_R(X))$, 
fix a prime $\p\in\min(\Supp_R(X))$. It follows that $X_{\p}\not\simeq 0$, so we have $\supp_{R_{\p}}(X_{\p})\neq\emptyset$.
Thus, there is a prime $\q\in\spec(R)$ with $\q\subseteq\p$ and $\q R_{\p}\in\supp_{R_{\p}}(X_{\p})$.
From the next sequence
$$\Lotimes{\kappa(\q)}{X}\simeq\Lotimes[R_{\p}]{\kappa(\q R_{\p})}{X_{\p}}\not\simeq 0$$
it follows that $\q\in\supp_R(X)\subseteq\Supp_R(X)$. 
Since $\q\subseteq\p$, the minimality of $\p$ in $\Supp_R(X)$ implies that $\p=\q\in\supp_R(X)$.
From the containment $\supp_R(X)\subseteq\Supp_R(X)$, the fact that $\p$ is minimal in $\Supp_R(X)$ implies
that it is also minimal in $\supp_R(X)$.

For the reverse containment, fix a prime $\p\in\min(\supp_R(X))\subseteq\Supp_R(X)$. 
Suppose that $\p$ is not minimal in $\Supp_R(X)$, so there is a prime $\q\in\Supp_R(X)$ such that $\q\subsetneq\p$.
By assumption, we have $X_{\q}\not\simeq 0$, so there is a prime $\mathfrak r\in\spec(R)$ such that
$\mathfrak r\subseteq\mathfrak q$ and $\mathfrak r_{\q}\in\supp_{R_{\q}}(X_{\q})$.
As in the previous paragraph, this implies that $\mathfrak r\in\supp_R(X)$, so the minimality of $\p$ implies that
$\p=\mathfrak r\subseteq\q\subsetneq\p$, a contradiction.
\end{proof}

\begin{prop}\label{prop130528c}
Let $X\in\catd(R)$. 
\begin{enumerate}[\rm(a)]
\item\label{prop130528c1}
Then one has $\Supp_{R}(X) \subseteq \VE(\fa)$ if and only if $\supp_{R}(X) \subseteq \VE(\fa)$.
\item\label{prop130528c2}
The Zariski closures of $\Supp_{R}(X)$ and $\supp_{R}(X)$ are equal.
\end{enumerate}
\end{prop}

\begin{proof}
\eqref{prop130528c1}
The forward implication is by the containment $\supp_{R}(X) \subseteq \Supp_{R}(X)$. 
For the converse, assume that $\supp_{R}(X) \subseteq \VE(\fa)$, and 
let $\p\in\Supp_R(X)$. It follows that $\p$ is contained in a minimal element $\q$ of $\Supp_R(X)$,
which is in $\supp_R(X)\subseteq\VE(\fa)$ by Proposition~\ref{prop140709b}.
In other words, we have $\fa\subseteq\q\subseteq\p$, so $\p\in\VE(\fa)$.

\eqref{prop130528c2}
The Zariski closures $\ol{\Supp_R(X)}$ and $\ol{\supp_R(X)}$ are of the form $\VE(\fb)$ and $\VE(\fc)$ for some ideals $\fb,\fc\subseteq R$.
Thus, part~\eqref{prop130528c1} implies that
$\ol{\supp_{R}(X)}$ and $\ol{\Supp_R(X)}$ are contained in each other.
\end{proof}

The next result is dual to Proposition~\ref{fact130611a}, with some restrictions on the complexes involved.

\begin{prop}\label{lem140121a}
If $X\in\catd_-(R)$ and $M\in\catdf_+(R)$, then 
$$\supp_{R}(\Rhom{M}{X}) = \supp_R(M) \bigcap \supp_{R}(X).$$
\end{prop}

\begin{proof}
Let $\p\in\spec(R)$. Our assumptions on $X$ and $M$ explain the first isomorphism in the next sequence.
\begin{align*}
\Rhom[R_{\p}]{\kappa(\p)}{\Rhom MX_\p}
&\simeq\Rhom[R_{\p}]{\kappa(\p)}{\Rhom[R_\p]{M_\p}{X_\p}}\\
&\simeq\Rhom[R_{\p}]{\Lotimes[R_\p]{\kappa(\p)}{M_\p}}{X_\p}\\
&\simeq\Rhom[R_{\p}]{\Lotimes[\kappa(\p)]{\kappa(\p)}{(\Lotimes[R_\p]{\kappa(\p)}{M_\p})}}{X_\p}\\
&\simeq\Rhom[\kappa(\p)]{\Lotimes[R_\p]{\kappa(\p)}{M_\p}}{\Rhom[R_\p]{\kappa(\p)}{X_\p}}
\end{align*}
The remaining isomorphisms are Hom-tensor adjointness and tensor cancellation.
Since $\p\in\supp_R(\Rhom MX)$ if and only if $\Rhom[R_{\p}]{\kappa(\p)}{\Rhom MX_\p}\not\simeq 0$,
the above isomorphisms imply that 
$\p\in\supp_R(\Rhom MX)$ if and only if $\Rhom[\kappa(\p)]{\Lotimes[R_\p]{\kappa(\p)}{M_\p}}{\Rhom[R_\p]{\kappa(\p)}{X_\p}}\not\simeq 0$.
The K\"unneth formula tells us that this is so if and only if $\Lotimes[R_\p]{\kappa(\p)}{M_\p}\not\simeq 0\not\simeq\Rhom[R_\p]{\kappa(\p)}{X_\p}$,
that is, if and only if $\p\in\supp_R(M)\bigcap\supp_R(X)$.
\end{proof}

\section{Co-support}\label{sec140110a}

In this section, we study co-support for complexes \emph{\`a la}~\cite{benson:csc}. 
Our main result here is Proposition~\ref{prop140111c}.
It is worth noting that our notion of small co-support is related to minimal  flat resolutions of modules in a manner similar to
the relation between small support and minimal injective resolutions from Proposition~\ref{lem130622a}; 
see~\cite{enochs:dbm}.

\begin{defn}\label{defn130529a}
Let $X\in\catd(R)$.
\begin{enumerate}[(a)]
\item
The ``small''  co-support of $X$ is 
$$\cosupp_R(X)=\{ \mathfrak{p} \in \spec(R)\mid \Rhom{\kappa(\mathfrak{p})}{X} \not\simeq 0\}.$$  
\item
The ``large'' co-support of $X$ is 
$$\Cosupp_R(X)=\{\mathfrak{p} \in \spec(R)\mid \Rhom{R_\p}{X}\not\simeq 0\}.$$
\end{enumerate}
\end{defn}

\begin{fact}\label{defn130503aq}
If $M$ is an $R$-module, then
\begin{align*}
\cosupp_R(M)&=\{ \mathfrak{p} \in \operatorname{Spec}(R)\mid \Ext i{\kappa(\p)} M\neq 0 \hspace{.05in}  \text{for some} \hspace{.05in} i \}\\
\operatorname{Co-supp}_R(M) &= \{ \mathfrak{p} \in \operatorname{Spec}(R)\mid \Ext i{R_\p} M\neq 0 \hspace{.05in}  \text{for some} \hspace{.05in} i\}.
\end{align*}
\end{fact}

\begin{fact}\label{fact140109ay}
Assume that $(R,\m,k)$ is local, and let $X\in\catd(R)$.
Then Fact~\ref{prop140111a}.
implies that $\m\in\supp_R(X)$ if and only if $\m\in\cosupp_R(X)$.
See Proposition~\ref{prop140111e}\eqref{prop140111e2} for a significant improvement on this. 
\end{fact}

Our next result is a version of Proposition~\ref{prop140111d} for co-support.

\begin{prop}\label{prop140111c}
Let $X\in\catd(R)$ and  $\p\in\spec(R)$, and let  $\x=x_1,\ldots,x_n$ be a generating sequence for $\p$.
Then the following conditions are equivalent:
\begin{enumerate}[\rm(i)]
\item \label{prop140111c1}
$\Lotimes[R_{\p}]{\kappa(\p)}{\Rhom{R_{\p}}{X}} \not\simeq 0$;
\item \label{prop140111c2}
$\Lotimes{K^{R_\p}(\x)}{\Rhom{R_{\p}}X}\not\simeq 0$;
\item \label{prop140111c3}
$\mathbf{L}\Lambda^{\p R_{\p}}(\Rhom{R_{\p}}X)\not\simeq 0$, that is, $\mathbf{L}\Lambda^{\p}(\Rhom{R_{\p}}X)\not\simeq 0$;
\item \label{prop140111c4}
$\p\in\cosupp_R(X)$, i.e., $\Rhom{\kappa(\mathfrak{p})}{X} \not\simeq 0$;
\item \label{prop140111c5}
$\Rhom{K^{R_\p}(\x)}{X}\not\simeq 0$;
\item \label{prop140111c6}
$\mathbf{R}\Gamma_{\p R_{\p}}(\Rhom{R_{\p}}X)\not\simeq 0$, that is, $\mathbf{R}\Gamma_{\p}(\Rhom{R_{\p}}X)\not\simeq 0$;
\item \label{prop140111c7}
$\p R_\p\in\cosupp_{R_\p}(\Rhom{R_\p}X)$.
\end{enumerate}
\end{prop}

\begin{proof}
Apply Fact~\ref{prop140111a} to the $R_{\p}$-complex $\Rhom{R_\p}X$ as in the proof of Proposition~\ref{prop140111d}.
For instance, we have
$\Lotimes[R_{\p}]{\kappa(\p)}{\Rhom{R_{\p}}{X}} \not\simeq 0$
if and only if
$\Rhom[R_{\p}]{\kappa(\p)}{\Rhom{R_{\p}}{X}} \not\simeq 0$. 
In light of the isomorphisms
\begin{align*}
0
&\not\simeq\Rhom[R_{\p}]{\kappa(\p)}{\Rhom{R_{\p}}{X}}\\
&\simeq\Rhom{\Lotimes[R_\p]{R_{\p}}{\kappa(\p)}}{X}\\
&\simeq\Rhom{\kappa(\p)}{X}
\end{align*}
we have 
$\Lotimes[R_{\p}]{\kappa(\p)}{\Rhom{R_{\p}}{X}} \not\simeq 0$
if and only if
$\Rhom{\kappa(\p)}{X}\not\simeq 0$.
\end{proof}

\begin{disc}\label{rmk140708b}
The equivalence of conditions~\eqref{prop140111c4} and~\eqref{prop140111c6}
in Proposition~\ref{prop140111c} show that our definition
of $\cosupp_R(X)$ is equivalent to that from~\cite{benson:csc}; see~\cite[Remark~4.17]{benson:csc}.
\end{disc}

\begin{cor}\label{fact140109ax}
If $X\in\catd(R)$, then
$\cosupp_R(X)\subseteq\Cosupp_R(X).$
\end{cor}

\begin{proof}
If $\p\in\cosupp_R(X)$, then 
$\Lotimes[R_{\p}]{\kappa(\p)}{\Rhom{R_{\p}}{X}} \not\simeq 0$
by Proposition~\ref{prop140111c},
so $\Rhom{R_{\p}}{X} \not\simeq 0$,
as desired.
\end{proof}

Our next result compares to part of Fact~\ref{fact140109a}.
Note that the sets $\supp_R(X)$ and  $\cosupp_R(X)$ have maximal elements, since $R$ is noetherian.

\begin{prop}\label{prop140111e}
Let $X\in\catd(R)$.
\begin{enumerate}[\rm(a)]
\item \label{prop140111e1}
We have $\cosupp_R(X)=\emptyset$ if and only if $X\simeq 0$.
\item \label{prop140111e2}
The sets $\supp_R(X)$ and  $\cosupp_R(X)$ have the same maximal elements with respect to containment.
\end{enumerate}
\end{prop}

\begin{proof}
In view of Remark~\ref{rmk140708b}, this follows from~\cite[Theorems 4.5 and 4.13]{benson:csc}.
\end{proof}

The next results are proved like 
Propositions~\ref{prop140710g}--\ref{prop130528c}.
Several of these are in~\cite{benson:csc}, though our proofs are more direct.

\begin{prop}\label{prop140710f}
Given a distinguished triangle $X\to Y\to Z\to$ in $\catd(R)$
one has 
$$\cosupp_R(Y)\subseteq\cosupp_R(X)\bigcup\cosupp_R(Z).$$
\end{prop}

\begin{prop}\label{prop140710b}
Given a set $\{X(i)\}_{i\in\Lambda}\subseteq\catd(R)$
one has 
$$\cosupp_R\left(\prod_iX(i)\right)=\bigcup_i\cosupp_R(X(i))\subseteq\cosupp_R\left(\coprod_iX(i)\right).$$
\end{prop}

\begin{prop}\label{cor130602a}
If $X ,Y\in\catd(R)$, then 
$$\operatorname{co-supp}_{R}(\Rhom{X}{Y})= \supp_{R}(X)\bigcap\cosupp_R(Y).$$
\end{prop}

\begin{prop}\label{lem130611ax}
If $X\in\catd(R)$, then 
$$\cosupp_{R}(\LL{a}{N}) = \cosupp_{R}(N) \bigcap \VE(\fa).$$
\end{prop}

\begin{prop}\label{prop140709c}
Let $X\in\catd(R)$.
Then the sets $\cosupp_R(X)$ and $\Cosupp_R(X)$ have the same minimal elements with respect to containment, that is,
we have $\min(\cosupp_R(X))=\min(\Cosupp_R(X))$. 
\end{prop}

\begin{prop}\label{prop130528cc}
Let $X\in\catd(R)$. 
\begin{enumerate}[\rm(a)]
\item\label{prop130528cc1}
Then one has $\Cosupp_{R}(X) \subseteq \VE(\fa)$ if and only if $\cosupp_{R}(X) \subseteq \VE(\fa)$.
\item\label{prop130528cc2}
The Zariski closures of $\Cosupp_{R}(X)$ and $\cosupp_{R}(X)$ are equal.
\end{enumerate}
\end{prop}

\section{Morphisms}\label{sec140710a}

In this section, we study some consequences of support conditions for morphisms in $\catd(R)$. 
In particular, we prove
Theorem~\ref{thm130318ax} from the introduction; see Theorems~\ref{thm130318aqq} and~\ref{thm130318aq}.
These results are key for our work in~\cite{sather:ascfc}, e.g., for our version of Foxby equivalence
in the adic setting.

\begin{lem}\label{lem130318aqq}
Let $X,M\in\catd(R)$  with $\supp_R(X)\subseteq\supp_R(M)$.
Then  the following conditions are equivalent:
\begin{enumerate}[\rm(i)]
\item \label{lem130318aqqi}
$X\simeq 0$;
\item \label{lem130318aqqii}
$\Rhom MX\simeq 0$;
\item \label{lem130318aqqiii}
$\Lotimes MX\simeq 0$.
\end{enumerate}
\end{lem}

\begin{proof}
It suffices to prove the implications \eqref{lem130318aqqii}$\implies$\eqref{lem130318aqqi} and 
\eqref{lem130318aqqiii}$\implies$\eqref{lem130318aqqi}.

\eqref{lem130318aqqiii}$\implies$\eqref{lem130318aqqi}.
Proposition~\ref{fact130611a} explains the next sequence:
$$\supp_{R}(\Lotimes{X}{M})=\supp_{R}(X) \bigcap \supp_{R}(M)=  \supp_{R}(X).$$ 
Thus, we have $X\simeq 0$ if and only if $\Lotimes XM\simeq 0$ by Fact~\ref{fact140109a}.

\eqref{lem130318aqqii}$\implies$\eqref{lem130318aqqi}.
Assume that $X\not\simeq 0$.
This implies that $\supp_R(X)\neq\emptyset$, so let $\p$ be a maximal element of $\supp_R(X)$. 
It follows that $\p\in\supp_R(M)$ by assumption. 
Proposition~\ref{prop140111e}\eqref{prop140111e2} implies that $\p$ is in $\cosupp_R(X)$,
hence in $$\supp_R(M)\cap\cosupp_R(X)=\cosupp_R(\Rhom MX).$$
We conclude that $\supp_R(\Rhom MX)\neq\emptyset$, hence $\Rhom MX\not\simeq 0$.
\end{proof}

Our next result contains part of Theorem~\ref{thm130318ax} from the introduction.

\begin{thm}\label{thm130318aqq}
Let $M\in\catd(R)$, and let
$f\colon Y \to Z$ be morphism in $\catd(R)$ with $\supp_R(Y), \supp_R(Z) \subseteq \supp_R(M)$.
The following conditions are equivalent:
\begin{enumerate}[\rm(i)]
\item \label{thm130318aqqi}
$f$ is an isomorphism in $\catd(R)$;
\item \label{thm130318aqqii}
$\Rhom{M}{f}$ is an isomorphism in $\catd(R)$;
\item \label{thm130318aqqiii}
$\Lotimes{M}{f}$ is an isomorphism in $\catd(R)$.
\end{enumerate}
\end{thm}

\begin{proof}
\eqref{thm130318aqqi}$\iff$\eqref{thm130318aqqiii}.
There is a distinguished triangle
$$Y\xra fZ\to X\to$$
in $\catd(R)$.
The condition $\operatorname{supp}(Y), \operatorname{supp}(Z) \subseteq \supp_R(M)$ implies that
$\operatorname{supp}(X) \subseteq\supp_R(M)$.
Thus,  we have $X\simeq 0$ if and only if $\Lotimes{M}X\simeq 0$ by Lemma~\ref{lem130318aqq}.

Also, the above triangle gives rise to another distinguished triangle
\begin{gather*}
\Lotimes{M}Y\xra{\Lotimes{M}f}\Lotimes{M}Z\to \Lotimes{M}X\to 
\end{gather*}
in $\catd(R)$. 
The morphism $f$ is an isomorphism in $\catd(R)$ if and only if $X\simeq 0$; and
$\Lotimes{M}f$ is an isomorphism in $\catd(R)$ if and only if $\Lotimes{M}X\simeq 0$.
Thus, the desired equivalence follows from the previous paragraph.

\eqref{thm130318aqqi}$\iff$\eqref{thm130318aqqii}.
This is handled similarly using $\Rhom M-$.
\end{proof}

\begin{cor}\label{thm130318aqqc}
Let 
$f\colon Y \to Z$ be morphism in $\catd(R)$ with $\supp_R(Y), \supp_R(Z) \subseteq\VE(\fa)$, 
and let $K$ be the Koszul complex on a generating sequence for $\fa$.
Then  the following conditions are equivalent:
\begin{enumerate}[\rm(i)]
\item \label{thm130318aqqci}
$f$ is an isomorphism in $\catd(R)$;
\item \label{thm130318aqqcii}
$\Rhom{K}{f}$ is an isomorphism in $\catd(R)$;
\item \label{thm130318aqqciii}
$\Rhom{R/\fa}{f}$ is an isomorphism in $\catd(R)$;
\item \label{thm130318aqqciiii}
$\LL{a}{f}$ is an isomorphism in $\catd(R)$;
\item \label{thm130318aqqciiz}
$\Lotimes{K}{f}$ is an isomorphism in $\catd(R)$;
\item \label{thm130318aqqciiiz}
$\Lotimes{R/\fa}{f}$ is an isomorphism in $\catd(R)$;
\item \label{thm130318aqqciiiiz}
$\RG{a}{f}$ is an isomorphism in $\catd(R)$.
\end{enumerate}
\end{cor}

\begin{proof}
We have $\supp_R(K)=\supp_R(R/\fa)=\supp_R(\RG aR)=\VE(\fa)$ by
Fact~\ref{fact140109a} and Proposition~\ref{lem130611a}.
Since $\RG a-\simeq \Lotimes{\RG aR}-$ by Fact~\ref{fact130619b'}, condition~\eqref{thm130318aqqci}
is equivalent to each of the conditions~\eqref{thm130318aqqcii}--\eqref{thm130318aqqciiiiz} by Theorem~\ref{thm130318aqq}.
\end{proof}

Our next result recovers part of~\cite[Corollary 5.7(1)]{benson:lcstc}.
For $X\in\catd_-(R)$, it can be proved using Proposition~\ref{lem130611a}.
See, e.g., \cite{dwyer:cmtm, yekutieli:ccc, yekutieli:hct, yekutieli:sccmc} for more on these complexes.

\begin{prop}\label{prop130619ad}
Let $X\in\catd(R)$.
Then one has $\supp_R(X)\subseteq V(\fa)$
if and only if the natural morphism $\RG aX\to X$
is an isomorphism in $\catd(R)$.
\end{prop}

\begin{proof}
Assume that $\supp_R(X)\subseteq V(\fa)$. The Corollary
to~\cite[Theorem (0.3)*]{lipman:lhcs} implies that the induced morphism $\RG af\colon\RG a{\RG aX}\to \RG aX$ is an isomorphism in $\catd(R)$. 
Since we have $\supp_R(X),\supp_R(\RG aX)\subseteq V(\fa)$ by Proposition~\ref{lem130611a},
we conclude from Theorem~\ref{thm130318aqq} that $f$ is an isomorphism.

The converse follows from Proposition~\ref{lem130611a}.
\end{proof}

\begin{prop}\label{lem130805a}
Let $X,Y\in\catd(R)$  such that $\supp_R(X),\supp_R(Y)\subseteq V(\fa)$.
If $\LL aX\simeq\LL a Y$ in $\catd(R)$, then $X\simeq Y$.
\end{prop}

\begin{proof}
By Proposition~\ref{prop130619ad}, the support conditions on $X$ and $Y$ explain the first and last isomorphisms in the next sequence:
$$X\simeq \RG aX\simeq\RG a{\LL aX}\simeq\RG a{\LL aY}\simeq \RG aY\simeq Y.$$
The second and fourth ones are by part (iv) of the Corollary to~\cite[Theorem~(0.3)*]{lipman:lhcs}.
The third one is by assumption.
\end{proof}

The next  results are versions of~\ref{lem130318aqq}--\ref{lem130805a} for co-support, with similar proofs. 
Note that Theorem~\ref{thm130318aq} contains the rest of Theorem~\ref{thm130318ax} from the introduction.

\begin{lem}\label{lem130318aq}
Let $X,M\in\catd(R)$  with $\cosupp_R(X)\subseteq\supp_R(M)$.
Then  the following conditions are equivalent:
\begin{enumerate}[\rm(i)]
\item \label{lem130318aqi}
$X\simeq 0$;
\item \label{lem130318aqii}
$\Rhom MX\simeq 0$;
\item \label{lem130318aqiii}
$\Lotimes MX\simeq 0$.
\end{enumerate}
\end{lem}

\begin{thm}\label{thm130318aq}
Let $M\in\catd(R)$, and let
$f\colon Y \to Z$ be morphism in $\catd(R)$ with $\cosupp_R(Y), \cosupp_R(Z) \subseteq \supp_R(M)$.
The following conditions are equivalent:
\begin{enumerate}[\rm(i)]
\item \label{thm130318aqi}
$f$ is an isomorphism in $\catd(R)$;
\item \label{thm130318aqii}
$\Rhom{M}{f}$ is an isomorphism in $\catd(R)$;
\item \label{thm130318aqiii}
$\Lotimes{M}{f}$ is an isomorphism in $\catd(R)$.
\end{enumerate}
\end{thm}

\begin{cor}\label{thm130318aqc}
Let 
$f\colon Y \to Z$ be morphism in $\catd(R)$, 
and let $K$ be the Koszul complex on a generating sequence for $\fa$.
Assume that $\cosupp_R(Y), \cosupp_R(Z) \subseteq\VE(\fa)$.
Then  the following conditions are equivalent:
\begin{enumerate}[\rm(i)]
\item \label{thm130318aqci}
$f$ is an isomorphism in $\catd(R)$;
\item \label{thm130318aqcii}
$\Rhom{K}{f}$ is an isomorphism in $\catd(R)$;
\item \label{thm130318aqciii}
$\Rhom{R/\fa}{f}$ is an isomorphism in $\catd(R)$;
\item \label{thm130318aqciiii}
$\LL{a}{f}$ is an isomorphism in $\catd(R)$;
\item \label{thm130318aqciiz}
$\Lotimes{K}{f}$ is an isomorphism in $\catd(R)$;
\item \label{thm130318aqciiiz}
$\Lotimes{R/\fa}{f}$ is an isomorphism in $\catd(R)$;
\item \label{thm130318aqciiiiz}
$\RG{a}{f}$ is an isomorphism in $\catd(R)$.
\end{enumerate}
\end{cor}

Our next result recovers part of~\cite[Corollary 4.8]{benson:csc}.
For perspective, note that if $X$ is a finitely generated $R$-module, then the natural morphism $X\to\LL a X$
is an isomorphism in $\catd(R)$ if and only if $X$ is $\fa$-adically complete. 
See, e.g., \cite{dwyer:cmtm, yekutieli:ccc, yekutieli:hct, yekutieli:sccmc} for more on these complexes.

\begin{prop}\label{prop130619ae}
Let $X\in\catd(R)$.
Then one has $\cosupp_R(X)\subseteq V(\fa)$
if and only if the natural morphism $X\to\LL a X$
is an isomorphism in $\catd(R)$.
\end{prop}

\begin{prop}\label{lem130805aq}
Let $X,Y\in\catd(R)$ with $\cosupp_R(X),\cosupp_R(Y)\subseteq V(\fa)$.
If $\RG aX\simeq\RG a Y$ in $\catd(R)$, then $X\simeq Y$.
\end{prop}

We conclude this section with other versions of~\ref{lem130318aqq}--\ref{thm130318aqq}, with similar proofs. 

\begin{lem}\label{lem130318aqm}
Let $X,M\in\catd(R)$  such that either $\supp_R(X)\subseteq\cosupp_R(M)$ or $\cosupp_R(X)\subseteq\cosupp_R(M)$.
Then 
$X\simeq 0$
if and only if $\Rhom XM\simeq 0$.
\end{lem}

\begin{thm}\label{thm130318aqm}
Let $M\in\catd(R)$, and let
$f\colon Y \to Z$ be morphism in $\catd(R)$ such that either $\supp_R(Y), \supp_R(Z) \subseteq \cosupp_R(M)$
or $\cosupp_R(Y), \cosupp_R(Z) \subseteq \cosupp_R(M)$.
Then
$f$ is an isomorphism in $\catd(R)$
if and only if
$\Rhom f{M}$ is an isomorphism in $\catd(R)$.
\end{thm}

\section{Some co-support Computations}\label{sec140710b}

Since co-support is (to us) somewhat mysterious, we devote this section to some computations.
(See also the discussion at the end of~\cite[Section 4]{benson:csc}.)
We begin with the co-support of Matlis duals.
Recall that an injective $R$-module $E$ is \emph{faithfully injective}
if for all $R$-modules $M$ one has $M=0$ if and only if $\Hom ME=0$.
For example, the direct sum $\bigoplus_\m E_R(R/\m)$ is faithfully injective,
where the sum is indexed over all maximal ideals $\m$ of $R$.

\begin{prop}\label{prop140118a}
If $E$ is a faithfully injective $R$-module and  $X\in\catd(R)$,
then  $\cosupp_R(E)=\spec(R)$ and
$\cosupp_R(\Rhom XE)=\supp_R(X)$.
\end{prop}

\begin{proof}
As $E$ is faithfully injective, for all $\p\in\spec(R)$, we have
$\Rhom{\kappa(\p)}{E}\not\simeq 0$.
Thus, the conclusion $\cosupp_R(E)=\spec(R)$ follows by definition.
Because of this, Proposition~\ref{cor130602a} explains the computation of
$\operatorname{co-supp}_{R}(\Rhom{X}{E})$.
\end{proof}

The next three results extend the first half of the previous one. Note that every injective $R$-module decomposes
uniquely into the  form given in Proposition~\ref{prop140710e}.

\begin{lem}\label{lem140710b}
Given a set $\{J(i)\}_{i\in\Lambda}$ of injective $R$-modules,
one has 
$$\cosupp_R\left(\coprod_iJ(i)\right)=\bigcup_i\cosupp_R(J(i)).$$
\end{lem}

\begin{proof}
The containment $\cosupp_R\left(\coprod_iJ(i)\right)\supseteq\bigcup_i\cosupp_R(J(i))$
is from Proposition~\ref{prop140710b}. For the reverse containment, 
note that the injective module $\coprod_iJ(i)$ is a summand of $\prod_iJ(i)$. Hence, we also have
$$\cosupp_R\left(\coprod_iJ(i)\right)\subseteq\cosupp_R\left(\prod_iJ(i)\right)=\bigcup_i\cosupp_R(J(i))$$
as desired.
\end{proof}

\begin{prop}\label{prop140710d}
Given a prime $\p\in\spec(R)$, one has
$$\cosupp_R(E_R(R/\p))=\{\q\in\spec(R)\mid\q\subseteq\p\}.$$
\end{prop}

\begin{proof}
For one containment, let $\q\in\spec(R)$ such that $\q\subseteq\p$.
The natural map $R\to\kappa(\q)$ factors as the composition of the following natural maps:
$R\to R_\p\to R_\q\to\kappa(\q)$.
Note that $E_R(R/\p)=E_{R_\p}(\kappa(\p))$ is faithfully injective over $R_{\p}$, hence the first step in the next sequence:
\begin{align*}
0
&\neq\Hom[R_\p]{\kappa(\q)}{E_R(R/\p)}
=\Hom{\kappa(\q)}{E_R(R/\p)}.
\end{align*}
The second step is a standard property of $R_\p$-modules. It follows that 
we have $\q\in\cosupp_R(E_R(R/\p))$, as desired.

For the reverse containment, let $\q\in\spec(R)$ such that $\q\not\subseteq\p$.
We need to show that $\q\notin\cosupp_R(E_R(R/\p))$.
The condition $\q\not\subseteq\p$ implies that $\Gamma_{\q}(E_R(R/\p))=0$ because every element
of $\q\ssm\p$ acts as a unit on $E_R(R/\p)$.
The module $\kappa(\q)$ is $\q$-torsion, so we have
\begin{align*}
\Hom{\kappa(\q)}{E_R(R/\p)}
&\cong\Hom{\kappa(\q)}{\Gamma_{\q}(E_R(R/\p))}=0.
\end{align*}
Since $E_R(R/\p)$ is injective, this implies that $\Rhom{\kappa(\q)}{E_R(R/\p)}\simeq 0$,
so we have $\q\notin\cosupp_R(E_R(R/\p))$, as desired.
\end{proof}

\begin{prop}\label{prop140710e}
Let $I=\coprod_{\p\in\Lambda}E_R(R/\p)^{(\mu_p)}$
for some index set $\Lambda\subset\spec(R)$ and exponent sets $\mu_p\neq\emptyset$.
Then we have
$$\cosupp_R(I)=\bigcup_{\p\in\Lambda}\cosupp_R(E_R(R/\p)^{(\mu_p)})=\{\q\mid\text{$\q\subseteq\p$ for some $\p\in\Lambda$}\}.$$
\end{prop}

\begin{proof}
The desired equalities are consequences of the next sequence
\begin{align*}
\cosupp_R(I)
&=\bigcup_{\p\in\Lambda}\cosupp_R(E_R(R/\p)^{(\mu_p)})\\
&=\bigcup_{\p\in\Lambda}\cosupp_R(E_R(R/\p))\\
&=\{\q\mid\text{$\q\subseteq\p$ for some $\p\in\Lambda$}\}
\end{align*}
which follow from
Lemma~\ref{lem140710b} and Proposition~\ref{prop140710d}.
\end{proof}

For the next two results, recall that an $R$-complex $C\in\catdfb(R)$ is \emph{semidualizing}
if the natural homothety morphism $R\to\Rhom CC$ is an isomorphism in $\catd(R)$.
A \emph{dualizing} $R$-complex is a semidualizing $R$-complex of finite injective dimension.

\begin{prop}\label{prop140118b}
For a semidualizing $R$-complex $C$,
one has $\supp_R(C)=\spec(R)$ and  $\cosupp_R(C)=\cosupp_R(R)$.
\end{prop}

\begin{proof}
Since $C\in\catdfb(R)$, we have
$\supp_R(C)=\Supp_R(C)\subseteq\spec(R)$
by Fact~\ref{fact140109a}.
On the other hand, for each $\p\in\spec(R)$, one has
$$0\not\simeq R_{\p}\simeq\Rhom CC_{\p}\simeq\Rhom[R_{\p}]{C_{\p}}{C_{\p}}.$$
Thus, $C_{\p}\not\simeq 0$ and so $\p\in\Supp_R(C)$.

With the isomorphism $\Rhom CC\simeq R$, this implies that
\begin{align*}
\cosupp_R(C)
&=\spec(R)\bigcap\cosupp_R(C)\\
&=\supp_R(C)\bigcap\cosupp_R(C)\\
&=\cosupp_R(\Rhom CC)\\
&=\cosupp_R(R)
\end{align*}
by Proposition~\ref{cor130602a}.
\end{proof}

The next result is Theorem~\ref{prop140118dx}\eqref{prop140118dxa} from the introduction.

\begin{thm}\label{prop140118d}
If $R$ has a dualizing complex,
then   each $X\in\catdfb(R)$ has
$$\cosupp_R(X)=\supp_R(X)\bigcap\cosupp_R(R)\subseteq\supp_R(X).$$
\end{thm}

\begin{proof}
Let $D$ be a dualizing $R$-complex.
As we have $X\in\catdfb(R)$, Grothendieck duality implies that
$\Rhom XD\in\catdfb(R)$ and
\begin{equation}\label{eq140719a}
X\simeq\Rhom{\Rhom XD}D
\end{equation}
in $\catd(R)$.
It follows from Propositions~\ref{lem140121a} and~\ref{prop140118b} that
$$\supp_R(\Rhom XD)
=\supp_R(X)\bigcap\supp_R(D)
=\supp_R(X).
$$
With Proposition~\ref{prop140118b}, this explains the third equality in the next sequence.
\begin{align*}
\cosupp_R(X)
&=\cosupp_R(\Rhom{\Rhom XD}D)\\
&=\supp_R(\Rhom XD)\bigcap\cosupp_R(D)\\
&=\supp_R(X)\bigcap\cosupp_R(R)
\end{align*}
The other equalities are from the isomorphism~\eqref{eq140719a} and Proposition~\ref{cor130602a}.
\end{proof}

Note that Proposition~\ref{cor140118a} shows that one can have proper containment
or equality in the next result.

\begin{thm}\label{prop140118g}
For each $X\in\catdfb(R)$, one has
$\cosupp_R(X)\subseteq\supp_R(X)$.
\end{thm}

\begin{proof}
If $X\simeq 0$, then $\cosupp_R(X)=\emptyset=\supp_R(X)$, and we are done.
So, assume that $X\not\simeq 0$.
Set $i=\inf(X)$ and $s=\sup(X)$.
Then, we have
\begin{align*}
\supp_R(X)
&=\Supp_R(X)\\
&=\bigcup_{j=i}^s\Supp_R(\HH_j(X))\\
&=\bigcup_{j=i}^s\VE(\ann_R(\HH_j(X)))\\
&=\VE\left(\bigcap_{j=i}^s\ann_R(\HH_j(X))\right).
\end{align*}
Set $\fa=\bigcap_{j=i}^s\ann_R(\HH_j(X))$, so we have $\supp_R(X)=\VE(\fa)$.
Now, each module $\HH_j(X)$ is annihilated by $\fa$.
In particular, each $\HH_j(X)$ is $\fa$-adically complete.
So the natural morphism $X\to \LL aX$ is an isomorphism in $\catd(R)$
by \cite[Theorem~1.21]{yekutieli:ccc}.
Thus, Proposition~\ref{prop130619ae} implies that $\cosupp_R(X)\subseteq\VE(\fa)=\supp_R(X)$.
\end{proof}

The next example shows that the assumption $X\in\catdfb(R)$ in the previous result is essential.

\begin{ex}\label{ex140118a}
Assume that $(R,\m)$ is local and not artinian, and set $E=E_R(k)$.
In particular, we have 
$$\supp_R(E)=\{\m\}\subsetneq\spec(R)=\cosupp_R(E)$$
by Propositions~\ref{lem130622a} and~\ref{prop140118a}.
\end{ex}

In light of Proposition~\ref{prop140118b} and Theorem~\ref{prop140118d}, it is natural to ask for a characterization of $\cosupp_R(R)$.
The next results with the discussion at the end of~\cite[Section 4]{benson:csc} 
show that this is a subtle question:
for instance, we can have $\cosupp_R(R)\neq\supp_R(R)$ or $\cosupp_R(R)=\supp_R(R)$.

\begin{fact}[\protect{\cite[Proposition 4.19]{benson:csc}}]\label{prop140118c}
The following conditions are equivalent:
\begin{enumerate}[(i)]
\item $R$ is $\fa$-adically complete;
\item $\cosupp_R(R)\subseteq\VE(\fa)$;
\item for all $X\in\catdfb(R)$, one has $\cosupp_R(X)\subseteq\VE(\fa)$.
\end{enumerate}
Proposition~\ref{prop140111e}\eqref{prop140111e2} implies that
$\max(\supp_R(X))\subseteq\cosupp_R(X)$ for all $X\in\catd(R)$.
Thus, if $(R,\m)$ is local, then the following conditions are equivalent:
\begin{enumerate}[(i)]
\item $R$ is $\m$-adically complete;
\item $\cosupp_R(R)=\{\m\}$;
\item for all $0\not\simeq X\in\catdfb(R)$, one has $\cosupp_R(X)=\{\m\}$.
\end{enumerate}
\end{fact}

\begin{prop}\label{cor140118a}
If $(R,\m)$ is a 1-dimensional local integral domain, then 
$$\cosupp_R(R)=\begin{cases}
\{\m\}&\text{if $R$ is $\m$-adically complete, and}\\
\spec(R)&\text{if $R$ is not $\m$-adically complete.}\end{cases}$$
\end{prop}

\begin{proof}
Proposition~\ref{prop140111e}\eqref{prop140111e2}  
implies that 
$$\{\m\}=\mspec(R)\subseteq\cosupp_R(R)\subseteq\spec(R)=\{0,\m\}$$
and
Fact~\ref{prop140118c} says that
$R$ is $\m$-adically complete if and only if $\cosupp_R(R)=\{\m\}$.
This provides the desired conclusion.
\end{proof}

The next result is Theorem~\ref{prop140118dx}\eqref{prop140118dxb} from the introduction.
It applies, for instance, to any polynomial ring in one variable over a field and its localizations.

\begin{thm}\label{prop140709a}
If $R$ is a 1-dimensional  integral domain. If $R$ has a dualizing complex, then 
$$\cosupp_R(R)=\begin{cases}
\mspec(R)&\text{if $R$ is local and  complete, and}\\
\spec(R)&\text{otherwise.}\end{cases}$$
\end{thm}

\begin{proof}
If $R$ is local, then we are done by Proposition~\ref{cor140118a}.
So, we assume for the rest of the proof that $R$ is not local.
Let $Q=\kappa(0)$ denote the field of fractions of $R$, and let $D$ be a dualizing $R$-complex.
Shift $D$ if necessary to assume that its minimal injective resolution 
has the form
$$0\to Q\xra\partial E\to 0$$
where $E:=\bigoplus_{\m}E_R(R/\m)$;
here the direct sum is taken over all $\m\in\mspec(R)$, and the complex is concentrated in degrees 0 and $-1$.
Proposition~\ref{prop140118b} shows that it suffices to show that $\cosupp_R(D)=\spec(R)$.

Since $\supp_R(D)=\spec(R)$,  Proposition~\ref{prop140111e}\eqref{prop140111e2} implies that
$$\mspec(R)\subseteq\cosupp_R(D)\subseteq\spec(R)=\mspec(R)\bigcup\{0\}.$$
Thus, it suffices to show that $0\in\cosupp_R(D)$.
By Proposition~\ref{prop140709c}, it suffices to show that $0\in\Cosupp_R(D)$, that is, that
$\Rhom QD\not\simeq 0$. We accomplish this by showing that $\Ext 1QD\neq 0$.
Using the injective resolution of $D$ from the beginning of this proof,
we see that $\Ext 1QD$ is the cokernel of the following $Q$-linear map.
$$\underbrace{\Hom QQ}_{\cong Q}\xra{\Hom Q\partial}\Hom QE$$
To show that this cokernel is non-zero, it suffices to show  $\rank_Q(\Hom QE)\geq 2$.

To this end, Proposition~\ref{prop140710e} implies that, for each $\m\in\mspec(R)$, we have $0\in\cosupp_R(E_R(R/\m))$.
This means that 
$$\Hom{Q}{E_R(R/\m)}=\Hom{\kappa(0)}{E_R(R/\m)}\neq 0.$$
Since $R$ is not local we can write $E\cong E_R(R/\m_1)\bigoplus E_R(R/\m_2)\bigoplus E'$
where $\m_1,\m_2$ are distinct maximal ideals. It follows that 
\begin{align*}
\rank_Q(\Hom QE)
&\geq\sum_{i=1}^2 \rank_Q(\Hom Q{E_R(R/\m_i)})\geq 2
\end{align*}
by the previous display.
\end{proof}

\begin{disc}\label{rmk140709a}
Let 
$\mathfrak{C}(R)$ denote the set of ideals $\fa\subsetneq R$ such that
$R$ is $\fa$-adically complete. Since $R$ is $0$-adically complete, this is a non-empty set of ideals of $R$,
so the noetherian property implies that $\mathfrak{C}(R)$ has maximal elements.
If $R$ is $\fa$-adically complete and $\fb$-adically complete, then it is also $(\fa+\fb)$-adically complete;
so $\mathfrak{C}(R)$ has a unique maximal element, that we denote $\mathfrak{c}(R)$.
Since $R$ is also $\rad{\mathfrak{c}(R)}$-adically complete, the maximality of $\mathfrak{c}(R)$ implies that
$\mathfrak{c}(R)$ is a radical ideal.
\end{disc}

\begin{question}\label{q140118a}
With the above notation, must we have $\cosupp_R(R)=\VE(\mathfrak{c}(R))$?
\end{question}

\begin{disc}\label{rmk140118c}
Fact~\ref{prop140118c} implies that
$\cosupp_R(R)\subseteq\VE(\mathfrak{c}(R))$.

Question~\ref{q140118a} has an affirmative answer for 1-dimensional local domains
(where we have $\mathfrak{c}(R)=0$ or $\mathfrak{c}(R)=\m$) by Proposition~\ref{cor140118a}.
And Fact~\ref{prop140118c} shows the same for complete local rings (where $\mathfrak{c}(R)=\m$).
Theorem~\ref{prop140709a} does the same for 1-dimensional  domains with dualizing complexes, as follows. 

The local case is already established, so assume that $R$ is not local. Then it suffices to show that $\mathfrak{c}(R)=0$,
that is, that $R$ is not $\fa$-adically complete for any $\fa\neq 0$. 
Suppose that $\fa\neq 0$ and $R$ is $\fa$-adically complete. 
It follows that $\fa$ is contained in the Jacobson radical  $\Jac (R)$.
Since $R$ is a 1-dimensional domain, the non-zero ideal $\fa$ has a primary decomposition
$\fa=\bigcap_{i=1}^n\q_i$ such that each $\rad{\q_i}$ is maximal.
It follows that $\mspec(R)=\{\rad{\q_1},\ldots,\rad{\q_n}\}$
and that $R$ is complete with respect to $\Jac(R)$. This implies that $R$ is a product of complete local rings.
However, the fact that $R$ is a domain implies that it does not decompose as a non-trivial product.
Hence, $R$ is a complete local ring, contradicting the assumption that $R$ is not local.
\end{disc}

On the subject of products, we have the following. It shows that 
Question~\ref{q140118a} has an affirmative answer for a  product $R=A\times B$ if and only if 
it has an affirmative answer for the factors $A$ and $B$.
A simple induction argument extends this to finite products, thus reducing the question to the case of rings
that do not admit non-trivial product decompositions.

\begin{prop}\label{prop140709d}
Let $A$ and $B$ be non-zero commutative  noetherian rings, and set $R=A\times B$.
\begin{enumerate}[\rm(a)]
\item \label{prop140709d1}
One has $\fc(R)=\fc(A)\oplus\fc(B)$.
\item \label{prop140709d2}
Identifying $\spec(R)$ with the disjoint union $\spec(A)\bigsqcup\spec(B)$,
one has an equality $\cosupp_R(R)=\cosupp_A(A)\bigsqcup\cosupp_B(B)$.
\item \label{prop140709d3}
One has $\cosupp_R(R)=\VE(\mathfrak{c}(R))$ if and only if $\cosupp_A(A)=\VE(\mathfrak{c}(A))$ and $\cosupp_B(B)=\VE(\mathfrak{c}(B))$.
\end{enumerate}
\end{prop}

\begin{proof}
The ideals of $R$ are of the form $I\oplus J$ for ideals $I\subseteq A$ and $J\subseteq B$.
It is straightforward to show that $R$ is $(I\oplus J)$-adically complete if and only if
$A$ is $I$-adically complete and $B$ is $J$-adically complete. 
This explains part~\eqref{prop140709d1}.

The prime ideals of $R$ are of the form $P\oplus B$ and $A\oplus Q$ for primes $P\subset A$ and $Q\subset B$.
It is straightforward to show that $\kappa(P\oplus B)\cong\kappa(P)$ and
$$\Rhom{\kappa(P\oplus B)}{R}\simeq\Rhom[A]{\kappa(P)}{A}.$$
Thus, we have $P\oplus B\in \cosupp_R(R)$ if and only if $P\in\cosupp_A(A)$,
and similarly for $A\oplus Q$. 
This explains part~\eqref{prop140709d2}.

Lastly, part~\eqref{prop140709d3} follows from parts~\eqref{prop140709d1} and~\eqref{prop140709d2}.
\end{proof}

We close this section with a flagrant display of how little we understand about Question~\ref{q140118a} and about
$\cosupp_R(R)$ in general. Note that for the rings in this question, a straightforward cardinality argument shows that
$\fc(R)=0$.

\begin{question}
Let $k$ be a field, and let $R$ be the polynomial ring $k[X,Y]$ or the localized polynomial ring $k[X,Y]_{(X,Y)}$, with field of fractions $Q$.
Do we have $\cosupp_R(R)=\spec(R)$? In particular, do we have $0\in\cosupp_R(R)$, that is, do we have $\Rhom QR\not\simeq 0$?
\end{question}

\section{Adic Finiteness}\label{sec140109c}

The main result of this section is Theorem~\ref{thm130612a}, i.e., Theorem~\ref{cor130612a} from the introduction.
We begin with versions for half-bounded complexes in Propositions~\ref{prop130610a} and~\ref{prop130613b}.
It should be noted that, in~\ref{prop130613b},  the equivalence of conditions (i) and (iii)--(vi) is
in~\cite[Claim 1]{kawasaki:ccc}. However, our proof  is significantly different in a key way:
instead of using spectral sequences, we use a small amount of technology from differential graded (DG) homological algebra.
Specifically, we use the following.

Let $\x=x_1,\ldots,x_n\in R$. 
The Koszul complex $K=K^R(\x)$ has the structure of a positively graded, commutative DG $R$-algebra.
As with $R$-complexes, we index DG $K$-modules homologically, 
and $\Lotimes[K]--$ and $\Rhom[K]--$ are the derived functors of $\Otimes[K]--$ and $\Hom[K]--$.
References on DG algebras and DG modules include~\cite{apassov:hddgr, avramov:ifr, avramov:lgh, 
avramov:dgha, beck:sgidgca, felix:rht, frankild:ddgmgdga, nasseh:lrfsdc, nasseh:ldgm}. 
We most closely follow the conventions from~\cite{nasseh:lrfsdc}. 

\begin{prop}\label{prop130610a}
Let $X\in\catd_+(R)$. Then the following conditions are equivalent:
\begin{enumerate}[\rm(i)]
\item\label{prop130610a1}
One has $\Lotimes{(R/\fa)}{X}\in\catd^{\text{f}}(R)$;
\item\label{prop130610a2}
One has $\Lotimes{(R/\fb)}{X}\in\catd^{\text{f}}(R)$ for all ideals $\fb \supseteq \fa$;
\item\label{prop130610a2'}
One has $\Lotimes{(R/\p)}{X}\in\catd^{\text{f}}(R)$ for all prime ideals $\p \in \VE(\fa)$;
\item\label{prop130610a3}
One has $\Lotimes{N}{X}\in\catd^{\text{f}}(R)$ for all finitely generated $R$-modules $N$ such that $\Supp_{R}(N) \subseteq \VE(\fa)$;
\item\label{prop130610a2''}
One has $\Lotimes{(R/\rad{\fa)}}{X}\in\catd^{\text{f}}(R)$;
\item\label{prop130610a4}
One has $\Lotimes{Y}{X}\in\catd^{\text{f}}(R)$ for all  $Y\in\catdfb(R)$ with $\Supp_{R}(Y) \subseteq \VE(\fa)$;
\item\label{prop130610a5}
One has $\Lotimes{K^R(\underline{x})}{X}\in\catd^{\text{f}}(R)$ for some (equivalently, for every) generating sequence $\underline{x}$ of $\fa$.
\end{enumerate}
\end{prop}

\begin{proof}
$\eqref{prop130610a1}\implies\eqref{prop130610a2}$. Consider the following commutative diagram of ring epimomorphisms.
$$\xymatrix{
R \ar[r] \ar[rd] & R/\fa \ar[d] \\
& R/\fb
}$$
By assumption, we have $\Lotimes{(R/\fa)}{X}\in\catdf_+(R)$, so  $\Lotimes{(R/\fa)}{X}\in\catdf_+(R/\fa)$.
Using a degree-wise finite free resolution $F$ of $\Lotimes{(R/\fa)}{X}$ over $R/\fa$, we see that that $\Lotimes{(R/\fb)}{X}\simeq \Lotimes[R/\fa]{(R/\fb)}{(\Lotimes{R/\fa}{X})}\in\catd^{\text f}(R/\fb)$, so $\Lotimes{(R/\fb)}{X}$ is in $\catd^{\text f}(R)$. 

$\eqref{prop130610a2}\implies\eqref{prop130610a2'}$. trivial.

$\eqref{prop130610a2'}\implies\eqref{prop130610a3}$. Assume that $N$ is finitely generated with $\Supp_{R}(N) \subseteq \VE(\fa)$. Then there is a prime filtration 
$0=N_0 \subseteq N_1 \subseteq \cdots \subseteq N_t=N$ such that $N_{i}/N_{i-1} \cong R/\p_{i}$ and $\p_i \in \Supp{N}\subseteq\VE(\fa)$ for $i=1,\ldots,t$. We argue by induction on $t$. 

Base case: $t=1$. Then $N \cong N_{1}/N_{0} \cong R/\p$, where $\fa \subseteq \p$. Then by assumption $\Lotimes{N}{X}\simeq\Lotimes{(R/\p)}{X}\in\catdf(R)$.

Induction step. Assume that $\Lotimes{N}{X}\in\catdf(R)$ for all finitely generated $R$-modules $N$ with $\Supp_{R}(N) \subseteq \VE(\fa)$  having a prime filtration of length $t-1$. 
Let $N$ have a prime filtration $0=N_0 \subseteq N_1 \subseteq \cdots \subseteq N_t=N$. Consider the short exact sequence $0 \to N_{t-1} \to N \to N/N_{t-1} \to 0$. 
Applying $\Lotimes{-}{X}$, we obtain the following distinguished triangle in $\catd(R)$.
$$\Lotimes{N_{t-1}}{X} \to \Lotimes{N}{X} \to \Lotimes{(N/N_{t-1})}{X} \to$$ 
By the induction hypothesis
and base case, we have $\Lotimes{N_{m-1}}{X}$ and $\Lotimes{(N/N_{m-1})}{X}$ in $\catd^{\text f}(R)$. Therefore, a long exact sequence argument shows that $\Lotimes{N}{X}\in\catd^{\text f}(R)$. 

$\eqref{prop130610a3}\implies\eqref{prop130610a2''}$. trivial.

$\eqref{prop130610a2''}\implies\eqref{prop130610a2'}$. This follows from the implication
$\eqref{prop130610a1}\implies\eqref{prop130610a2'}$ (applied to the ideal $\rad\fa$) since $\VE(\fa)=\VE(\rad\fa)$.

$\eqref{prop130610a3}\implies\eqref{prop130610a4}$. Assume that $Y\in\catdfb(R)$ with $\Supp_{R}(Y) \subseteq \VE(\fa)$. Then we have $\Supp_{R}(\HH_{i}(Y)) \subseteq \VE(\fa)$. 
By assumption, each module $\HH_i(Y)$ is finitely generated, so we have $\Lotimes{\HH_{i}(Y)}{X}\in\catd^{\text f}(R)$ for all $i$. 
We proceed by induction on $\amp(Y)$. 

Base case: $\amp(Y)=0$. Then $Y$ has one non-zero homology module, so we have $Y \simeq \shift^{i} \HH_{i}(Y)$ for some $i$. 
It follows that $\Lotimes{Y}{X}\simeq\shift^i\Lotimes{\HH_{i}(Y)}{X}\in\catd^{\text f}(R)$ by the previous paragraph.

Induction step: Assume that  for all $Y'\in\catdfb(R)$ such that $\amp(Y') < \amp(Y)$ and $\Supp_{R}(Y') \subseteq \VE(\fa)$
we have $\Lotimes YX\in\catd^{\text f}(R)$. 
Let $s=\sup(Y)$. 
From a ``soft truncation'' of $Y$,  there is a distinguished triangle
$$\shift^{s} \HH_{s}(Y) \to Y \to Y'' \to
$$
in $\catd(R)$
such that the induced map $\HH_{i}(Y)\to \HH_{i}(Y'')$ is an isomorphism  for all $i<s$ and $\HH_{i}(Y'')= 0$ for all $i\geq s$. 
Thus, the induction hypothesis applies to $Y''$ and the base case applies for $\shift^{s} \HH_{s}(Y)$,
so we have $\Lotimes {Y''}YX,\Lotimes {\shift^{s} \HH_{s}(Y)}X\in\catd^{\text f}(R)$.
A long exact sequence argument yields the desired conclusion $\Lotimes {Y}X\in\catd^{\text f}(R)$.

$\eqref{prop130610a4}\implies\eqref{prop130610a5}$.  Condition (v) is the special case $X=K^R(\underline{x})$ of (iv).

$\eqref{prop130610a5}\implies\eqref{prop130610a1}$.  Set $K=K^R(\underline{x})$ and consider the following commutative diagram of morphisms of DG $R$-algebras.
$$\xymatrix{
R \ar[r] \ar[rd] & K \ar[d] \\
& R/\fa.
}$$
Since $\Lotimes{K}{X}\in\catdf(R)$, we have $\Lotimes{K}{X}\in\catdf(K)$.
Using
a degree-wise finite semi-free resolution of $\Lotimes{K}{X}$ over $K$, we conclude that the  complex 
$$\Otimes[K]{(R/\fa)}{X} \simeq \Lotimes[K]{(R/\fa)}{(\Lotimes{K}{X})}\in\catdf(R/\fa)$$
so we have $\Otimes[K]{(R/\fa)}{X} \in\catdf(R)$ as well.
\end{proof}

The next result is proved like Proposition~\ref{prop130610a}, using the functor $\Rhom -X$ in place of $\Lotimes -X$.

\begin{prop}\label{prop130613b}
Let $X\in\catd_-(R)$. Then the following conditions are equivalent:
\begin{enumerate}[\rm(i)]
\item
One has $\Rhom{R/\fa}{X}\in\catd^{\text{f}}(R)$;
\item
One has $\Rhom{R/\fb}{X}\in\catd^{\text{f}}(R)$ for all ideals $\fb \supseteq \fa$;
\item
One has $\Rhom{R/\p}{X}\in\catd^{\text{f}}(R)$ for all prime ideals $\p\in\VE(\fa)$;
\item
One has $\Rhom{N}{X}\in\catd^{\text{f}}(R)$ for all finitely generated $R$-modules $N$ such that $\Supp_{R}(N) \subseteq \VE(\fa)$;
\item
One has $\Rhom{R/\rad\fa}{X}\in\catd^{\text{f}}(R)$;
\item
One has $\Rhom{Y}{X}\in\catd^{\text{f}}(R)$ for all  $Y\in\catdfb(R)$ with $\Supp_{R}(Y) \subseteq \VE(\fa)$;
\item
One has $\Rhom{K^R(\underline{x})}{X}\in\catd^{\text{f}}(R)$ for some (equivalently, for every) generating sequence $\underline{x}$ of $\fa$.
\end{enumerate}
\end{prop}

\begin{disc}\label{rmk140112a}
In the previous result, the self-dual nature of the Koszul complex $K$
implies that condition~(vii) is equivalent to the following:
\begin{enumerate}[$(\text{vii}')$]
\item 
One has $\Lotimes{K^R(\underline{x})}{X}\in\catd^{\text{f}}(R)$ for some (equivalently, for every) generating sequence $\underline{x}$ of $\fa$.
\end{enumerate}
\end{disc}

The following result is Theorem~\ref{cor130612a} from the introduction.

\begin{thm}\label{thm130612a}
Let $X\in\catd_{\text b}(R)$. Then the following conditions are equivalent:
\begin{enumerate}[\rm(i)]
\item\label{thm130612a1}
One has $\Lotimes{K(\underline{x})}{X}\in\catdfb(R)$  for some (equivalently for every) generating sequence $\underline{x}$ of $\fa$;
\item\label{thm130612a2}
One has  $\Lotimes{R/\mathfrak{a}}{X}\in\catd^{\text{f}}(R)$;
\item\label{thm130612a3}
One has  $\Rhom{R/\mathfrak{a}}{X}\in\catd^{\text{f}}(R)$;
\item\label{thm130612a4}
One has $\LLL aX\in\catdfb(\Comp Ra)$.
\end{enumerate}
\end{thm}

\begin{proof}
Since $X\in\catdb(R)$, we have $\Lotimes{K^R(\x)}X\in\catdb(R)$, as well.
Thus, the equivalence of conditions~\eqref{thm130612a1}--\eqref{thm130612a3} is  from Propositions~\ref{prop130610a} and~\ref{prop130613b}. 

In preparation for the rest of the proof, we note that $\LL aX\in\catdb(R)$. Indeed, we have $X\in\catdb(R)$ by assumption, and
$\LL aX\simeq\Rhom{\RG aR}{X}$, so it suffices to recall that Fact~\ref{fact130619b} implies that
$\RG aR$ has finite projective dimension over $R$.
Since the homology modules of $\LL aX$ and $\LLL aX$ are isomorphic,
it follows that $\LLL aX$ is in $\catdfb(\Comp Ra)$ if and only if it 
is in $\catdf(\Comp Ra)$.
Furthermore,  the natural morphism $\LLL aX\to\LLL a{\LL aX}$ is an isomorphism in $\catd(\Comp Ra)$ by part (ii) of the Corollary
to~\cite[Theorem (0.3)*]{lipman:lhcs}. In the notation of~\cite{yekutieli:ccc}, this means that $\LL aX\in\mathsf{D}(\mathsf{Mod}R)^{\text{b}}_{\text{$\fa$-com}}$.
Next, \cite[Lemma 2.5]{shaul:tpdcf} provides the following isomorphism  over $\Comp Ra/\fa\Comp Ra\cong R/\fa$,
hence, over $R$ and $\Comp Ra$:
\begin{align*}
\Rhom{R/\fa}{X}
&\simeq\Rhom[\Comp Ra]{\Comp Ra/\fa\Comp Ra}{\LLL a{X}}.
\end{align*}

$\eqref{thm130612a3}\iff\eqref{thm130612a4}$.
From the above isomorphism, we know  that $\Rhom{R/\fa}{X}\in\catdf(R)$ if and only if  $\Rhom[\Comp Ra]{\Comp Ra/\fa\Comp Ra}{\LLL a{X}}\in\catdf(R)$.
As the homology modules of this complex are annihilated by $\fa$,
it is in $\catdf(R)$  if and only if it is in $\catdf(\Comp Ra)$.
We conclude from~\cite[Lemma 3.8]{yekutieli:ccc} that 
$\Rhom[\Comp Ra]{\Comp Ra/\fa\Comp Ra}{\LLL a{X}}\in\catdf(\Comp Ra)$
if and only if
$\LLL aX\in\catdf(\Comp Ra)$,
i.e., if and only $\LLL aX\in\catdfb(\Comp Ra)$, as desired.
\end{proof}

\begin{disc}
For $X\in\catd_b(R)$, the equivalent conditions in Theorem~\ref{thm130612a} 
are equivalent to the conditions in Propositions~\ref{prop130610a} and~\ref{prop130613b}; see the first paragraph of the proof of the theorem.
We resist the temptation to list these conditions explicitly. On the other hand, it is worth noting that 
the same reasoning shows that these conditions are also equivalent to the
following:
\begin{enumerate}[(i)]
\item One has $\LLL bX\in\catdfb(\Comp Rb)$ for all ideals $\fb\supseteq\fa$;
\item One has $\LLL pX\in\catdfb(\Comp Rp)$ for all prime ideals $\p\VE(\fa)$.
\end{enumerate}
\end{disc}

We are now prepared to define $\fa$-adic finiteness.

\begin{defn}\label{def120925d}
A complex $X\in\catdb(R)$ is \emph{$\mathfrak{a}$-adically finite} if it satisfies the equivalent conditions of Theorem~\ref{thm130612a} 
and $\operatorname{supp}_R(X) \subseteq \operatorname{V}(\mathfrak{a})$.
Let $\mathcal{C}_{\fa}(R)$ denote the full subcategory of $\catd(R)$ consisting of $\fa$-adically finite $R$-complexes.
\end{defn}

\begin{disc}\label{disc140712a}
Note that, in the next definition, the condition $\operatorname{supp}_R(X) \subseteq \operatorname{V}(\mathfrak{a})$
is equivalent to $\operatorname{Supp}_R(X) \subseteq \operatorname{V}(\mathfrak{a})$
by Proposition~\ref{prop130528c}\eqref{prop130528c1}.
In other words, this is equivalent to $\supp_R(\HH_i(X))\subseteq\VE(\fa)$ for all $i$. 
Thus, $X$ is $\fa$-adically finite if and only if it is in $\catdb(R)$ and $\fa$-cofinite, in the language of~\cite{hartshorne:adc};
see~\cite[Theorem 5.1]{hartshorne:adc}.
\end{disc}

\begin{prop}\label{prop130528a}
Let $X\in\catd_{\text b}(R)$.
\begin{enumerate}[\rm(a)]
\item\label{prop130528a1}
Then $X$ is $0$-adically finite if and only if $X\in\catdfb(R)$.
\item\label{prop130528a2}
Assume that $(R,\m,k)$ is local. 
Then $X$ is $\m$-adically finite if and only if each homology module $\HH_i(X)$ is artinian.
\end{enumerate}
\end{prop}

\begin{proof}
\eqref{prop130528a1}
This follows by definition, due to the isomorphism $\Rhom{R/(0)}{X}\simeq X$ and the equality $\VE(0)=\spec(R)$.

\eqref{prop130528a2}
For one implication, assume that each $\HH_i(X)$ is artinian.
It follows that for each prime ideal $\p\neq \m$, we have $\HH_i(X)_\p=0$, and hence $X_\p\simeq 0$.
In other words, we have $\Supp_R(X)\subseteq \VE(\m)$. By Proposition~\ref{prop130528c}\eqref{prop130528c1},
this implies that $\supp_R(X)\subseteq\VE(\m)$.
Furthermore, we have $\HH_i(X)\subseteq E^{(\mu_i)}$ for some integer $\mu^i$ where $E=E_R(k)$.
From the construction of  injective resolutions, say in~\cite[2.6.I]{foxby:hacr}, it follows that there is an injective
resolution $X\res I$ such that each $I_j$ is of the form $E^{(\lambda_j)}$ for some integer $\lambda_j$.
From this, we conclude that the complex $\Rhom kX\simeq\Hom kI$ is a bounded above complex of modules of the form
$\Hom k{I_j}\cong\Hom k{E^{(\lambda_j)}}\cong k^{(\lambda_j)}$.
In particular, each module $\HH_i(\Rhom kX)$ is a finite dimensional vector space over $k$,
so $X$ is $\m$-adically finite.

For the converse, assume that $X$ is $\m$-adically finite.
The condition $\supp_R(X)\subseteq\VE(\m)$ implies that 
the minimal injective resolution $X\res J$ consists of direct sums of copies of $E$ by Proposition~\ref{lem130622a}.
Moreover, we have $J_i\cong E^{(\mu_{i})}$ for each $i\in\bbz$, where
$\mu_i=\rank_k(\HH_i(\Rhom kX))<\infty$;
the finiteness is from the adic finiteness assumption on $X$.
Hence, $X$ is isomorphic in $\catd(R)$ to a complex of artinian $R$-modules, and it follows that
each of its homology modules is artinian, as desired.
\end{proof}

\begin{prop}\label{prop140112a}
The category $\mathcal{C}_{\fa}(R)$  is triangulated and thick.
\end{prop}

\begin{proof}
Let $K$ denote the Koszul complex over $R$ on a finite generating sequence for $\fa$.
By definition, this follows from the next straightforward facts:
\begin{enumerate}[1.]
\item  For each $X\in\catd_{\text b}(R)$ and each $i\in\bbz$,
the complex $X$ is $\fa$-adically finite if and only if $\shift^iX$ is $\fa$-adically finite.
\item
Given a distinguished triangle $X\to Y\to Z\to$ in $\catd(R)$,
if two of the three complexes $X,Y,Z$ are $\fa$-adically finite, then so is the third.
\item
For all $X,Y\in\catd_{\text b}(R)$, the direct sum
$X\oplus Y$ is $\fa$-adically finite if and only if 
$X$ and $Y$ are both $\fa$-adically finite.\qedhere
\end{enumerate}
\end{proof}

\begin{thm}\label{thm130330a}
Let $X$ be an $\fa$-adically finite $R$-complex. If $\fb$ is an ideal of $R$ such that $\fa \subseteq \fb$, then $\RG{b}{X}$ is $\fb$-adically finite.
\end{thm}
\begin{proof}
Note that we have $\supp(\RG{\fb}{X}) \subseteq \VE(\fb)$ by Proposition~\ref{lem130611a}. 
Since $X$ is $\fa$-adically finite, the complex $\Rhom{R/\fa}X$ is in $\catdf(R)$,
so Proposition~\ref{prop130613b} implies that 
$\Rhom{R/\fb}X\in\catdf(R)$.
The fact that $R/\fb$ is $\fb$-torsion implies that
$$\Rhom{R/\fb}X\simeq \Rhom{R/\fb}{\RG bX}$$ 
in $\catd(R)$.
Hence, the $R$-complex $\Rhom{R/\fb}{\RG bX}$ is in $\catdf(R)$, so $\RG bX$ is $\fb$-adically finite.
\end{proof}

We next provide an indication of how adic finiteness can give variations on previous results.
For instance, the next result compares to part of Fact~\ref{fact140109a}.

\begin{thm}\label{thm140722a}
Let $M\in\catdb(R)$ be $\fa$-adically finite.
Then one has
$\supp_R(M)=\Supp_R(M)$, and this set is Zariski-closed in $\spec(R)$.
\end{thm}

\begin{proof}
Let $K$ be the Koszul complex over $R$ on a generating sequence for $\fa$.
Since $M$ is $\fa$-adically finite, we have $\Lotimes KM\in\catdfb(R)$, so the first equality in the next sequence is by Fact~\ref{fact140109a}.
\begin{align*}
\Supp_R(\Lotimes KM)
&=\supp_R(\Lotimes KM)
=\supp_R(K)\bigcap\supp_R(M)
=\supp_R(M)
\end{align*}
The second equality is from Proposition~\ref{fact130611a} and the third equality follows from the conditions
$\supp_R(K)=\VE(\fa)\supseteq\supp_R(M)$; see Fact~\ref{fact140109a}.
Note that the condition $\Lotimes KM\in\catdfb(R)$ implies that $\Supp_R(\Lotimes KM)=\supp_R(M)$ is Zariski-closed in
$\spec(R)$; see
Fact~\ref{fact140109a}.
From this, we have
\begin{align*}
\ol{\supp_R(M)}
&=\supp_R(M)
\subseteq\Supp_R(M)
\subseteq\ol{\Supp_R(M)}
=\ol{\supp_R(M)}
\end{align*}
by Fact~\ref{fact140109a} and Proposition~\ref{prop130528c}\eqref{prop130528c2}.
\end{proof}

The next result compares to Propositions~\ref{fact130611a}, \ref{lem140121a}, and~\ref{cor130602a}.

\begin{thm}\label{fact130611ass}
If $X,M\in\catd_-(R)$ such that $M$ is $\fa$-adically finite, then 
$$\supp_{R}(\Rhom{M}{X}) \bigcap\VE(\fa)= \supp_R(M) \bigcap \supp_{R}(X).$$
\end{thm}

\begin{proof}
Let $\x=x_1,\ldots,x_n$ be a generating sequence for $\fa$, and set $K:=K^R(\x)$.
By assumption, the complex $\Lotimes KM$ is in $\catdfb(R)$, so Proposition~\ref{lem140121a}
explains the fourth equality in the next sequence:
\begin{align*}
\supp_{R}(\Rhom{M}{X}) \bigcap\VE(\fa)
&=\supp_R(\Lotimes K{\Rhom MX})\\
&=\supp_R(\Rhom K{\Rhom MX})\\
&=\supp_R(\Rhom{\Lotimes KM}{X})\\
&=\supp_R(\Lotimes KM)\bigcap\supp_R(X)\\
&=\supp_R(K)\bigcap\supp_R(M)\bigcap\supp_R(X)\\
&=\supp_R(M)\bigcap\supp_R(X)
\end{align*}
The first and fifth equalities are  by Proposition~\ref{fact130611a}; this uses the following conditions
$\supp_R(M)\subseteq\VE(\fa)=\supp_R(K)$, which also explain the last equality (see Fact~\ref{fact140109a}).
The second equality follows from the self-dual nature of the Koszul complex which manifests as the first isomorphism in the next sequence.
\begin{align*}
\Lotimes K{\Rhom MX}
\simeq\shift^n\Rhom K{\Rhom MX}
\end{align*}
The remaining equality  is from Hom-tensor adjointness.
\end{proof}

The next result augments~\cite[Corollary 5.8]{benson:lcstc}.

\begin{cor}\label{cor140722a}
If $X\in\catd_-(R)$ and let $M\in\catdb(R)$ be $\fa$-adically finite. Then one has 
$\Rhom{M}{X}\simeq 0$ if and only if  $\supp_R(M) \bigcap \supp_{R}(X)=\emptyset$.
\end{cor}

\begin{proof}
For one implication, assume that $\supp_R(M) \bigcap \supp_{R}(X)=\emptyset$.
Theorem~\ref{thm140722a} implies that $\supp_R(M)$ is Zariski-closed, so we have
$\overline{\supp_R(M)} \bigcap \supp_{R}(X)=\emptyset$.
The desired conclusion $\Rhom{M}{X}\simeq 0$ now follows from~\cite[Corollary 5.8]{benson:lcstc}.

For the converse, if $\Rhom{M}{X}\simeq 0$, then
Theorem~\ref{fact130611ass} implies that
\begin{align*}
\supp_R(M) \bigcap \supp_{R}(X)
&=\supp_{R}(\Rhom{M}{X}) \bigcap\VE(\fa)=\emptyset
\end{align*}
as desired.
\end{proof}

Compare the next result to Propositions~\ref{lem130611a} and~\ref{lem130611ax}.

\begin{cor}\label{cor140121a}
Let$X\in\catd_-(R)$. Then one has
$$\supp_{R}(\LL aX) \bigcap\VE(\fa)= \VE(\fa) \bigcap \supp_{R}(X).$$
Furthermore, we have $\LL aX\simeq 0$ if and only if $\VE(\fa)\bigcap\supp_R(X)=\emptyset$.
\end{cor}

\begin{proof}
The complex $M=\RG aR$ is $\fa$-adically finite by Theorem~\ref{thm130330a}, 
so Theorem~\ref{fact130611ass} implies that
\begin{align*}
\supp_{R}(\LL aX) \bigcap\VE(\fa)
&=\supp_{R}(\Rhom{\RG aR}{X}) \bigcap\VE(\fa)\\
&= \supp_R(\RG aR) \bigcap \supp_{R}(X)\\
&= \VE(\fa) \bigcap \supp_{R}(X)\\
\end{align*}
by Fact~\ref{fact130619b'}
and Proposition~\ref{lem130611a}.

From the isomorphism $\LL aX\simeq\Rhom{\RG aR}{X}$, 
Corollary~\ref{cor140722a} implies that
$\LL aX\simeq 0$ if and only if $\supp_R(\RG aR)\bigcap\supp_R(X)=\emptyset$,
that is, if and only if $\VE(\fa)\bigcap\supp_R(X)=\emptyset$, by Proposition~\ref{lem130611a}.
\end{proof}

Next, we give some examples to show that not every result for complexes in $\catdfb(R)$ extends to the $\fa$-adically finite situation.
The first one shows that one can have 
$\supp_{R}(\Rhom{M}{X}) \neq \supp_R(M) \bigcap \supp_{R}(X)$ in
Theorem~\ref{fact130611ass}, even if $R$ is $\fa$-adically complete.

\begin{ex}\label{ex140722a}
Let $(R,\m,k)$ be a local ring, and set $E=E_R(k)$. 
Then $E$ is $\m$-cofinite by Proposition~\ref{prop130528a}\eqref{prop130528a2}.
However, one has $\Rhom EE\simeq\Comp Rm$, so
\begin{gather*}
\supp_R(\Rhom EE)
=\supp_R(\Comp Rm)=\spec(R)\\
\supp_R(E)\bigcap\supp_R(E)
=\{\m\}.
\end{gather*}
See Proposition~\ref{lem130622a}.
If $R$ is not artinian, then we have $\spec(R)\neq\{\m\}$, so 
$\supp_R(\Rhom EE)\neq \supp_R(E)\bigcap\supp_R(E)$.
\end{ex}

If $M\in\catdfb(R)$ and $X,Y\in\catdb(R)$, then the evaluation morphisms
\begin{gather*}
\Lotimes{\Rhom MY}X\to\Rhom M{\Lotimes YX}\\
\Lotimes{M}{\Rhom YX}\to\Rhom{\Rhom MY}X
\end{gather*}
are isomorphisms under certain hypotheses, e.g., when $M,X,Y$ all have finite projective dimension and finite injective dimension.
The next example shows that this fails when $M$ is only assumed to be $\fa$-adicaly finite, even when $R$ is $\fa$-adically complete. 

\begin{ex}\label{ex140722b}
Let $k$ be a field, and consider the formal power series ring $R=k[\![T]\!]$ with $E=E_R(k)$ and $Q=E_R(R)$
(Note that $Q$  is the quotient field $k((T))$ of $R$.)
Proposition~\ref{lem130622a} implies that 
$\supp_R(E)=\{\m\}$ and $\supp_R(Q)=\{0\}$,
so we have $\Lotimes EQ=0$ by Proposition~\ref{fact130611a}.

From the minimal injective resolution of $R$
$$0\to R\to Q\to E\to 0$$
we have a distinguished triangle $R\to Q\to E\to$ in $\catd(R)$.
Apply $\Lotimes E-$ to obtain the next distinguished triangle.
$$\underbrace{\Lotimes ER}_{\simeq E}
\to\underbrace{\Lotimes EQ}_{\simeq 0}
\to\Lotimes EE
\to
$$
It follows that $\Lotimes EE\simeq \shift E$.
A similar computation shows  $\Rhom ER\simeq\shift^{-1}R$.

To finish the example, we compute:
\begin{gather*}
\Lotimes{\Rhom EE}E
\simeq\Lotimes RE\simeq E
\\
\Rhom E{\Lotimes EE}
\simeq\Rhom E{\shift E}\simeq\shift\Rhom EE\simeq\shift R
\\
\Lotimes{E}{\Rhom ER}
\simeq\Lotimes E{(\shift^{-1}R)}\simeq \shift^{-1} E
\\
\Rhom{\Rhom EE}R
\simeq\Rhom{R}R\simeq R.
\end{gather*}
Looking at the degrees of these complexes, it is clear that
$\Lotimes{\Rhom EE}E\not\simeq\Rhom E{\Lotimes EE}$
and
$\Lotimes{E}{\Rhom ER}\not\simeq\Rhom{\Rhom EE}R$.
\end{ex}

We end this paper with a criterion for an $\fa$-adically finite $R$-complex to satisfy the condition $\supp_R(M)=\VE(\fa)$.
It is  key for some of our work in~\cite{sather:ascfc}.

\begin{prop}\label{prop140121a}
If $M\in\catd(R)$ is $\fa$-adically finite such that $\Rhom MM\simeq\Comp Ra$,
then $\supp_R(M)=\VE(\fa)$.
\end{prop}

\begin{proof}
The containment $\supp_R(M)\subseteq\VE(\fa)$ holds by definition.
The reverse containment comes from the next sequence wherein the second step is from Theorem~\ref{fact130611ass}
and the fifth step is from Corollary~\ref{cor140121a}.
\begin{align*}
\supp_R(M)
&\supseteq\supp_R(M)\cap\VE(\fa)\\
&=\supp_R(M)\cap\supp_R(M)\cap\VE(\fa)\\
&=\supp_R(\Rhom MM)\cap\VE(\fa)\\
&=\supp_R(\Comp Ra)\cap\VE(\fa)\\
&=\supp_R(R)\cap\VE(\fa)\\
&=\VE(\fa).
\end{align*}
The other steps are straightforward or by definition.
\end{proof}

\section*{Acknowledgments}
We are grateful to Srikanth Iyengar for helpful conversations about this work,
and to the referee for thoughtful comments.

\providecommand{\bysame}{\leavevmode\hbox to3em{\hrulefill}\thinspace}
\providecommand{\MR}{\relax\ifhmode\unskip\space\fi MR }
\providecommand{\MRhref}[2]{%
  \href{http://www.ams.org/mathscinet-getitem?mr=#1}{#2}
}
\providecommand{\href}[2]{#2}

\end{document}